%% file: main.tex
\documentclass[12pt,a4paper,final]{article}

\RequirePackage[l2tabu, orthodox]{nag}
\usepackage{geometry}
\usepackage[english, strings]{babel}
\usepackage{nchairx}
\usepackage[utf8]{inputenc}
\usepackage[T1]{fontenc}       
\usepackage{longtable}         
\usepackage{exscale}           
\usepackage[sort]{cite}        
\usepackage{xspace}            
\usepackage{tikz}              
\usetikzlibrary{cd, decorations.pathmorphing}
\usepackage[multiuser]{fixme}  
\usepackage{ifdraft}           
\usepackage[expansion=false    
]{microtype}        
\usepackage[backref=page,      
final=true,         
pdfpagelabels,       
ocgcolorlinks, 
]{hyperref}         
\usepackage{ocgx2} 

\mathtoolsset{showmanualtags}

\usepackage{parskip}

\input{include/names-twist}
\input{include/Macros.tex} %

\title{Convergent Twist Deformations}

\date{}

\begin{document}

%
%

\maketitle

%
%
%

\begin{abstract}
    This paper establishes a functorial framework for convergence of
    Drinfeld's Universal Deformation Formula (UDF) on spaces of
    analytic vectors. This is accomplished by matching the order of
    the latter with an equicontinuity condition on the Drinfeld twist
    underlying the deformation. Throughout, we work with
    representations of finite-dimensional Lie algebras by continuous
    linear mappings on locally convex spaces. This allows us to
    establish not only convergence of the formal power series, but the
    continuity of the deformed bilinear mappings as well as the entire
    holomorphic dependence on the deformation parameter~$\hbar$.
    Finally, we demonstrate the effectiveness of our theory by
    applying it to the explicit Drinfeld twists constructed by
    Giaquinto and Zhang, where we establish both the equicontinuity
    condition and determine the corresponding spaces of analytic
    vectors for concrete representations. Thereby we answer a question
    posed by Giaquinto and Zhang whether a strict version of their
    formal twists is possible in the positive.
\end{abstract}

%
%

\newpage
\tableofcontents
\newpage

%
%
\section{Introduction}%
\label{sec:Introduction}%
\input{TeX/Introduction.tex}

\section*{Acknowledgements}
C.E. and M.H. have been supported from Italian Ministerial grant PRIN
2022 “Cluster algebras and Poisson Lie groups'', n. 20223FEA2E - CUP
D53C24003330006.  C.E. was partially supported by the National Group
for Algebraic and Geometric Structures, and their Applications (GNSAGA
-- INdAM). Moreover, the authors would like to thank Pierre Bieliavsky
for his helpful comments on related works, which have improved our
vision of the problems a great deal.

\section{Preliminaries}%
\label{sec:Preliminaries}%

Throughout the text, all Lie algebras are assumed to be
finite-dimensional of positive~dimension $d \in \N$ over
the field $\R$ of real numbers. For each Lie algebra $\liealg{g}$
we fix an auxiliary norm $\norm{\argument}$, and write
$\Ball_r(\xi)$ for the open ball of radius $r>0$ centered at
$\xi \in \liealg{g}$. We denote the open unit ball by
\begin{equation}
    \UnitBall
    \coloneqq
    \Ball_1(0)
    \subseteq
    \liealg{g}
\end{equation}
Moreover, we consider the universal enveloping algebra
$\Universal(\liealg{g})$ of $\liealg{g}$ and always invoke its
characteristic property to extend Lie algebra representations
\begin{equation}
    \varrho
    \colon
    \liealg{g}
    \longrightarrow
    \Linear(V)
\end{equation}
by continuous linear endomorphisms $\Linear(V)$ of a locally convex
space $V$ to algebra morphisms
\begin{equation}
    \varrho
    \colon
    \Universal(\liealg{g})
    \longrightarrow
    \Linear(V).
\end{equation}
Depending on the context, we use the notations
\begin{equation}
    \varrho(\Xi)
    \qquad \textrm{and} \qquad
    \Xi \acts \argument
\end{equation}
for the action of an element $\Xi \in \Universal(\liealg{g})$
and suppress the product of $\Universal(\liealg{g})$ in the
notation. Finally, we denote the set of all continuous seminorms
of a locally convex space $V$ by
\begin{equation}
    \label{eq:csDef}
    \cs(V)
    =
    \left\{
        \seminorm{p}\colon V \longrightarrow \field{R}
        \; \big| \;
        \seminorm{p}
        \textrm{ is a continuous seminorm}
    \right\}.
\end{equation}
For the convenience of the reader, we collect the most important symbols we shall
use throughout, some of which we have of course yet to define.

\section*{Notation}
\begin{description}[style=nextline,leftmargin=3.2cm]
    \item[$\liealg{g}$] Lie algebra of dimension $d < \infty$.
    \item[$V\formal{\hbar}$] Formal power series with coefficients
    in $V$ and formal parameter $\hbar$.
    \item[$\UnitBall \coloneqq \Ball_1(0) \subseteq \liealg{g}$] Open unit
    ball.
    \item[$\Linear(V)$] Continuous linear endomorphisms of a locally
    convex space $V$.
    \item[$\cs (V)$] Set of all continuous seminorms of a locally convex
    space $V$.
    \item[$\Universal(\liealg{g})$] Universal Enveloping Algebra of
    $\liealg{g}$
    \item[$\tensor$] Projective tensor product
    \item[$\categoryname{Triple}$] Category of $\liealg{g}$-triples.
    \item[$\categoryname{cTriple}$] Subcategory of continuous
    $\liealg{g}$-triples.
    \item[$\categoryname{mTriple}$] Subcategory of malleable
    $\liealg{g}$-triples.
    \item[$\categoryname{indTriple}$] Subcategory of inductive
    $\liealg{g}$-triples.
    \item[$\varrho$] Lie Algebra representation
    $\varrho \colon \liealg{g} \longrightarrow \Linear(V)$, denoted by
    $\acts$
    \item[$\Analytic_{R,r_0}(\varrho)$] Space of analytic vectors of order
    $R \ge 0$ and radius of convergence $\ge r_0$.
    \item[$ \Entire_R(\varrho)$] Space of entire vectors of order
    $R\geq 0$.
    \item[$\Analytic_R(\varrho)$] Space of analytic vectors of order
    $R \ge 0$.
\end{description}

\subsection{$\liealg{g}$-Triples}
\label{subsec:PreliminariesGTriples}%
\input{TeX/PreliminariesGTriples.tex}

\subsection{Analytic Vectors}
\label{subsec:PreliminariesAnalyticVectors}%
\input{TeX/PreliminariesAnalyticVectors.tex}

\subsection{The Infimum Argument for Projective Tensor Products}
\label{subsec:PreliminariesProjectiveTensor}%
\input{TeX/PreliminariesProjectiveTensors.tex}

\section{Convergence of Universal Deformation Formulas}
\label{sec:ConvergenceOfUDF}%

This section establishes our two main results on the universal
deformation of $\liealg{g}$-triples
\begin{equation}
    \mu
    \colon
    V \tensor W
    \longrightarrow
    X,
\end{equation}
namely the deformations of continuous $\liealg{g}$-triples of entire
vectors and of inductive $\liealg{g}$-triples of analytic vectors.

\subsection{Joint Continuity and Entire Vectors}
\label{subsec:ConvergenceOfUDFEntire}%
\input{TeX/ConvergenceOfUDFEntire.tex}

\subsection{Analytic Vectors of Inductive $\liealg{g}$-Triples}
\label{subsec:ConvergenceOfUDFAnalytic}%
\input{TeX/ConvergenceOfUDFAnalytic.tex}

\section{Examples}
\label{sec:Examples}%

In this section, we apply our abstract machinery to the explicit
Drinfeld twists constructed within~\cite{giaquinto.zhang:1998a}. This
boils down to the verification of the equicontinuity condition
\eqref{eq:FormalTwistEquicontinuityAnalytic} under the simplifications
from Remark~\ref{rem:SmallRadius}. To achieve this, the Cauchy
estimates we have derived within Lemma~\ref{lem:CauchyEstimates} are
our main tool.


\subsection{Abelian $\liealg{g}$-Triples}
\label{subsec:ExamplesAbelian}%
\input{TeX/ExamplesAbelian.tex}

\subsection{The $ax+b$ Lie Algebra}
\label{subsec:ExamplesAxPlusB}
\input{TeX/ExamplesAxPlusB.tex}

\subsection{An Abelian Extension of the Heisenberg Lie Algebra}
\label{subsec:ExamplesHeisenberg}
\input{TeX/ExamplesHeisenberg.tex}

\section{Outlook}
\label{subsec:Outlook}
\input{TeX/Outlook.tex}


%
\bibliographystyle{nchairx}
\phantomsection
\addcontentsline{toc}{section}{References}
\bibliography{include/Twists,dqbook,dqarticle,preprints}

%
%

\ifdraft{\clearpage}
\ifdraft{\phantomsection}
\ifdraft{\addcontentsline{toc}{section}{List of Corrections}}
\ifdraft{\listoffixmes}

%
%
\end{document}

%% file: include/names-twist.tex
\newcommand{\AuthorOne}{\textbf{Chiara Esposito}}
\newcommand{\AuthorTwo}{\textbf{Michael Heins}}
\newcommand{\AuthorThree}{\textbf{Stefan Waldmann}}

\newcommand{\AuthorAddressOne}{
        \begin{minipage}{8cm}
            \centering\small
            Dipartimento di Matematica\\
            Università degli Studi di Salerno\\
            via Giovanni Paolo II, 132\\
            84084 Fisciano (SA)\\
            Italy
        \end{minipage}
        \\
        }
\newcommand{\AuthorAddressTwo}{\AuthorAddressOne}
\newcommand{\AuthorAddressThree}{
        \begin{minipage}{8cm}
            \centering\small
            Department of Mathematics\\
            University of Würzburg\\
            Emil-Fischer-Straße 31\\
            Würzburg\\
            Germany
        \end{minipage}
        \\
}

\newcommand{\AuthorEmailOne}{\texttt{chesposito@unisa.it}}
\newcommand{\AuthorEmailTwo}{\texttt{michael.heins.mathematics@gmail.com}}
\newcommand{\AuthorEmailThree}{\texttt{stefan.waldmann@uni-wuerzburg.de}}

\author{\AuthorOne\thanks{\AuthorEmailOne}, \,
  \addtocounter{footnote}{2}
  \AuthorTwo\thanks{\AuthorEmailTwo},\\[0.2cm]
  \AuthorAddressTwo
  \\[-0.1cm]
  \AuthorThree\thanks{\AuthorEmailThree},\\[0.2cm]
  \AuthorAddressThree
}

%% file: include/Macros.tex
\makeatletter

\newcommand{\bibnote}[2]{\nocite{#1}\@namedef{#1chairxnote}{#2}}
\makeatother

\renewcommand{\tilde}{\widetilde}

\newcommand{\disk}{\mathbb{D}}

\newcommand{\C}{\field{C}}

\newcommand{\R}{\field{R}}

\newcommand{\N}{\field{N}}

\renewcommand{\epsilon}{\varepsilon}

\newcommand{\StirlingOne}[2]{\genfrac{[}{]}{0pt}{0}{#1}{#2}}

\newcommand{\Shuffle}{\group{Sh}}

\newcommand{\Linear}{\operatorname{L}}

\newcommand{\cs}{\operatorname{cs}}

\newcommand{\Entire}{\mathscr{E}}

\newcommand{\Analytic}{\mathscr{A}}
\renewcommand{\algebra}[1]{\mathfrak{#1}}

\newcommand{\UnitBall}{\mathbb{B}}

%% file: TeX/Introduction.tex

Formal deformation quantization has been introduced in
\cite{bayen.et.al:1977a,bayen.et.al:1978a} and its guiding principle
is to deform the commutative algebra of smooth functions $\Cinfty(M)$
on a Poisson manifold~$M$ into a noncommutative algebra
$\Cinfty(M)\formal{\hbar}$ of formal power series with respect to a
formal parameter $\hbar$ by means of a star product $\star$. This is
an associative product, $\mathbb{C}\formal{\hbar}$-bilinear, whose
zeroth-order term reproduces the usual pointwise multiplication, while
the first-order commutator recovers the Poisson bracket.  Existence of
such formal star products was first established for symplectic
manifolds \cite{dewilde.lecomte:1983b, fedosov:1994a}, and later for
arbitrary Poisson manifolds by Kontsevich~\cite{kontsevich:2003a}, see
also \cite{waldmann:2007a} for an introduction.

Despite the impressive impact of these results and their many
subsequent applications, genuine physical interpretations demand going
beyond purely formal power series, since the deformation parameter
$\hbar$ should be regarded as Planck's constant and not as a formal
parameter.  This motivates the interest for \emph{strict} versions of
deformation quantization.  A common route toward strictness replaces
formal deformations with C$^*$-algebraic deformations. This approach,
initiated by Rieffel (see in particular \cite{rieffel:1989a,
  rieffel:1993a}) and developed further by many others, see
e.g. \cite{bieliavsky.massar:2001b, bieliavsky:2002a,
  bieliavsky.bonneau.maeda:2007a, landsman:1998a}, typically relies on
oscillatory integral expressions for the star product, which allow
estimates strong enough to construct C$^*$-norms. While powerful, a
key obstacle is that there is, in general, no universal way to produce
star products via oscillatory integrals.

This leads to an alternative strategy \cite{waldmann:2014a} based on
functional analytic techniques: start from formal star products and
study their convergence properties directly. In several families of
examples this can work as follows. First, one identifies a ``small''
subalgebra of functions on which the star product converges for
relatively straightforward reasons.  Since general theorems are
lacking, one proceeds case by case. Next, one equips this subalgebra
with a suitable locally convex topology that makes the star product
continuous.  Again, this step is mostly example-driven. If successful,
completing the subalgebra then yields a larger locally convex
algebra---typically a Fr\'echet algebra---suitable for further
analysis.

In finite dimensions this program may not look more decisive than
earlier methods, as it still offers no general existence
results. Nevertheless, it covers different kinds of examples,
including analogues of algebras of unbounded operators. In addition,
it is well suited to infinite-dimensional situations, where
oscillatory integral techniques are generally unavailable. In this
sense, the approach complements the established strict deformation
quantizations by providing new and structurally different examples. A
detailed overview can be found in \cite{waldmann:2019a} and more
recent works include \cite{heins:2024a, heins.roth.waldmann:2023a,
  kraus.roth.schleissinger.waldmann:2023a:pre,
  esposito.schmitt.waldmann:2019a, kraus.roth.schoetz.waldmann:2019a,
  schoetz.waldmann:2018a, esposito.stapor.waldmann:2017a,
  beiser.waldmann:2014a, schmitt.schoetz:2022a,
  barmeier.schmitt:2022a, schmitt:2021a, pirkovskii:2025a}.

This paper focuses on some special star products, obtained from the
idea of Drinfeld of using symmetries to construct formal
deformations. More explicitly, given an action $\triangleright$ by
derivations of a Lie algebra $\liealg{g}$ on an associative algebra
$(\algebra{A},\mu_{\algebra{A}})$, we consider a \emph{Drinfeld twist}
\cite{drinfeld:1983a, drinfeld:1988a}, see also the textbook
\cite[Sec.~9.5]{etingof.schiffmann:1998a},
\begin{equation}
    \label{eq:DrinfeldTwist}
    F \in
    \bigl(
    \Universal(\liealg{g})
    \tensor
    \Universal(\liealg{g})
    \bigr)\formal{\hbar}
\end{equation}
defined by the following conditions:
\begin{enumerate}[label=\textit{\roman*.)}]
\item $F = 1 \tensor 1 + \sum_{n=1}^{\infty}\hbar^{n} \cdot F_{n}$.
\item
    $(F\tensor 1)(\Delta\tensor 1)(F) = (1\tensor F) (1\tensor
    \Delta)(F)$.
\item $(\varepsilon\tensor 1)F = (1\tensor \varepsilon)F = 1$.
\end{enumerate}
Here we denoted by $\Delta$ and $\varepsilon$ the standard coproduct
and counit that make the universal enveloping algebra
$\Universal(\liealg{g})$ of $\liealg{g}$ into a bialgebra.  Formal
twists allow us to obtain an associative formal deformation of
$\algebra{A}$ by means of a \emph{universal deformation formula} (UDF)
\begin{equation}
    \label{eq:udf}
    a \star_F b
    =
    \mu_{\algebra{A}}\bigl(F \triangleright (a \tensor b)\bigr)
    \qquad
    \textrm{for all }
    a,b
    \in
    \algebra{A}\formal{\hbar}.
\end{equation}
Here $\triangleright$ is the action of $\liealg{g}$ extended to the
universal enveloping algebra $\Universal(\liealg{g})$ and then to
$\Universal(\liealg{g})\tensor \Universal(\liealg{g})$ acting on
$\algebra{A}\tensor \algebra{A}$ in a diagonal fashion. Drinfeld
twists and the associated universal deformation formula continue to
attract considerable attention, with several recent contributions
exploring both structural aspects and concrete constructions; see, for
instance, \cite{bieliavsky.esposito.waldmann.weber:2018a,
  dandrea:2017a, esposito.schnitzer.waldmann:2017a,
  EspositoKleijn:2022}.

In this work, we address the question originally posed by Giaquinto
and Zhang \cite{giaquinto.zhang:1998a} regarding the existence of UDF
analogs that yield strict deformation quantizations. Our main abstract
results, presented in Section~\ref{sec:ConvergenceOfUDF}, are
structured as follows.  First, in
Theorem~\ref{thm:UniversalDeformationEntire}, we prove that assuming
the continuity of the deformation allows for the construction of
continuous universal deformations on spaces of entire vectors—those
analytic vectors with an infinite radius of convergence.  We then
extend our results by introducing the notion of \emph{malleable
  $\mathfrak{g}$-triples}. In
Theorem~\ref{thm:UniversalDeformationEntire} and
Theorem~\ref{thm:UniversalDeformationAnalytic}, we establish that for
the analytic vectors of such triples, the resulting series not only
converge but also define continuous multilinear mappings.  Both
constructions follow a rigorous methodology that ensures the
preservation of the $\mathfrak{g}$-triple structure when passing to
sub-spaces of well-behaved vectors, imposing an equicontinuity
condition on the twist.

Finally, in Section~\ref{sec:Examples}, we put the developed theory
into practice by applying it to the Drinfeld twists constructed by
Giaquinto and Zhang, demonstrating the effectiveness of our formalism
in non-formal contexts.

The paper is organized as follows: Section~\ref{sec:Examples} reviews
the preliminaries on $\mathfrak{g}$-triples, analytic vectors, and the
topology of projective tensor
products. Section~\ref{sec:ConvergenceOfUDF} develops the core
convergence results for UDFs on entire and analytic vectors.
Section~\ref{sec:Examples} concludes with concrete examples, including
Abelian $\mathfrak{g}$-triples and extensions of the Heisenberg Lie
algebra.  In Section~\ref{subsec:Outlook} we give a short outlook on
related questions, future research directions and additional
developments.

%% file: TeX/PreliminariesGTriples.tex

We begin by defining the central object we are going to work with.
\begin{definition}[$\liealg{g}$-triples]
    \label{def:gTriple}%
    Let $\liealg{g}$ be a Lie algebra.
    \begin{definitionlist}
    \item \label{item:gTriple}%
        A $\liealg{g}$-triple is a triple
        $(\varrho_V, \varrho_W, \varrho_X)$ of representations on
        locally convex spaces $(V,W,X)$ together with a linear mapping
        \begin{equation}
            \label{eq:gTriple}
            \mu
            \colon
            V \tensor W
            \longrightarrow
            X.
        \end{equation}
    \item \label{item:gTripleContinuity}%
        A $\liealg{g}$-triple is called continuous if the map
        \eqref{eq:gTriple} is continuous with respect to the
        projective tensor product topology on $V \tensor W$.
    \item \label{item:gTripleMalleability}%
        A $\liealg{g}$-triple is called malleable if
        \begin{equation}
            \label{eq:gTripleCompatibility}
            \xi
            \acts
            \mu(v \tensor w)
            =
            \mu
            \bigl(
            (\xi \acts v)
            \tensor
            w
            +
            v
            \tensor
            (\xi \acts w)
            \bigr)
        \end{equation}
        for all $v \in V$, $w \in W$ and $\xi \in \liealg{g}$.
    \end{definitionlist}
\end{definition}

For brevity and by slight abuse of language, we also refer to both the
triple $(V,W,X)$ and the mapping \eqref{eq:gTriple} as a
$\liealg{g}$-triple. Note that \eqref{eq:gTripleCompatibility} can be seen as a
generalization of the Leibniz rule when thinking of $\mu$ as a product
and a representation of $\xi \in \liealg{g}$ by derivations of
$\mu$.

Recall that a permutation $\sigma \in \group{S}_n$ is called a
$(k,n-k)$-shuffle for $0 \le k \le n$ if
\begin{equation}
    \label{eq:Shuffle}
    \sigma(1)
    <
    \cdots
    <
    \sigma(k)
    \qquad \textrm{and} \qquad
    \sigma(k+1)
    <
    \cdots
    <
    \sigma(n).
\end{equation}
We write $\Shuffle(k,n-k)$ for the set of all such permutations, of
which there are $\binom{n}{k}$ many.
\begin{lemma}
    Let
    \begin{equation}
        \label{eq:muIsBilinerMap}
        \mu
        \colon
        V \tensor W \longrightarrow X
    \end{equation}
    be a malleable $\liealg{g}$-triple. Then
    \begin{equation}
        \label{eq:MalleableSystemLeibniz}
        \xi_1 \acts \cdots \xi_n \acts
        \mu(v \tensor w)
        =
        \sum_{k=0}^{n}
        \sum_{\sigma \in \Shuffle(k,n-k)}
        \mu
        \bigl(
        (
        \xi_{\sigma(1)} \cdots \xi_{\sigma(k)}
        \acts
        v
        )
        \tensor
        (
        \xi_{\sigma(k+1)} \cdots \xi_{\sigma(n)}
        \acts
        w
        )
        \bigr)
    \end{equation}
    for all $v \in V$, $w \in W$ and
    $\xi_1, \ldots, \xi_n \in \liealg{g}$.
\end{lemma}
\begin{proof}
    This follows by a simple induction coupled with the observation
    that every $(k,n-k)$-shuffle $\sigma$ fulfills $\sigma(k) = k$ or
    $\sigma(k) = n$.
\end{proof}

Another way to interpret the above lemma is that the induced action of
any element $\Xi \in \Universal(\liealg{g})$ in the universal
enveloping algebra of $\liealg{g}$ satisfies the higher Leibniz rule
\begin{equation}
    \label{eq:HigherLeibnizXi}
    \Xi \acts \mu(v \tensor w)
    =
    \mu(\Xi_\sweedler{1} \acts v, \Xi_\sweedler{2} \acts w),
\end{equation}
where we use the standard Sweedler notation for the coproduct of
$\Universal(\liealg{g})$.

\begin{definition}[Morphisms of $\liealg{g}$-triples]
    \label{def:gTriplesMorphisms}%
    Let $\liealg{g}$ and $\tilde{\liealg{g}}$ be Lie algebras.
    \begin{definitionlist}
    \item A morphism between a $\liealg{g}$-triple and a
        $\tilde{\liealg{g}}$-triple
        \begin{equation}
            \label{eq:TwoBilinearMuTildeMu}
            \mu
            \colon
            V \tensor W \longrightarrow X
            \qquad
            \textrm{and}
            \qquad
            \tilde{\mu}
            \colon
            \tilde{V} \tensor \tilde{W}
            \longrightarrow
            \tilde{X}
        \end{equation}
        is a triple $T = (T_V,T_W,T_X)$ of linear mappings
        \begin{equation}
            \label{eq:MorphismTriple}
            T_V
            \colon
            V \longrightarrow \tilde{V},
            \quad
            T_W
            \colon
            W \longrightarrow \tilde{W},
            \quad
            T_X
            \colon
            X \longrightarrow \tilde{X}
        \end{equation}
        such that
        \begin{equation}
            \label{eq:gTriplesMorphism}
            T_X
            \bigl(
            \mu(v \tensor w)
            \bigr)
            =
            \tilde{\mu}
            \bigl(
            T_V v
            \tensor
            T_W w
            \bigr)
            \qquad
            \textrm{for all }
            v \in V,
            w \in W.
        \end{equation}

    \item A morphism of continuous $\liealg{g}$-triples is a morphism
        of $\liealg{g}$-triples consisting of continuous linear mappings.

    \item A morphism of malleable $\liealg{g}$-triples is a morphism
        of $\liealg{g}$-triples $(T_V, T_W, T_X)$ such that the
        mappings $T_Z$ are equivariant along a surjective Lie algebra
        morphism $\phi \colon \liealg{g} \longrightarrow \tilde{\liealg{g}}$,
        i.e.
        \begin{equation}
            \label{eq:gTriplesMorphismMalleability}
            T_Z(\xi \acts z)
            =
            \phi(\xi)
            \acts
            T_Z(z)
        \end{equation}
        for all $\xi \in \liealg{g}$, $Z \in \{V,W,X\}$ and $z \in Z$.
    \end{definitionlist}
\end{definition}

Combining Definition~\ref{def:gTriple} and
Definition~\ref{def:gTriplesMorphisms} indeed defines a category,
which we denote by $\categoryname{Triple}$. Imposing continuity,
malleability or both results in subcategories
$\categoryname{cTriple}$, $\categoryname{mTriple}$ and
$\categoryname{cmTriple}$. Fixing the Lie algebra $\liealg{g}$ and
demanding $\phi = \id_\liealg{g}$ yields further subcategories, which we
indicate by adding a subscript, e.g. by writing
$\categoryname{Triple}_\liealg{g}$.

%% file: TeX/PreliminariesAnalyticVectors.tex

Roughly speaking, the idea behind analytic vectors corresponding to a
Lie algebra representation is to implement the Lie correspondence by
demanding the convergence of the exponentiated representation when
applied to such a vector, at least for sufficiently small Lie algebra
elements. Making this precise results in the notion of analytic
vectors, where we moreover keep track of the resulting radius of
convergence with respect to our auxiliary norm. While analytic vectors
have been studied extensively since their inception within
\cite{cartier:1958a, nelson:1959a, garding:1960a}, see also the
textbooks \cite[Appendix~D]{taylor:1986a} and
\cite[Sec.~4.4]{warner:1972a} for more modern treatments, we will need
an additional parameter controlling the growth of their Taylor
coefficients, the \emph{order}.

The usual analytic vectors then correspond to order $R = 0$. Moreover,
negative order results in the Gevrey classes as comprehensively
studied within~\cite{terElst:1989a}, which are an important tool in
the theory of partial differential equations, and are an abstraction
of the function spaces Gevrey introduced within
\cite{gevrey:1918a}. In general, they no longer correspond to
holomorphic objects by the Denjoy-Carleman Theorem
\cite{carleman:1926a, denjoy:1956a}, and as such behave rather
differently.  We exclude them from our considerations in the sequel.
\begin{definition}[Analytic vectors I]
    \label{def:AnalyticVectors}%
    Let $R \ge 0$, $r_0 > 0$ and
    \begin{equation}
        \label{eq:VarrhoIsRepresentation}
        \varrho
        \colon
        \liealg{g}
        \longrightarrow
        \Linear(V)
    \end{equation}
    be a Lie algebra representation on a locally convex space $V$.
    \begin{definitionlist}
    \item For every $\seminorm{q} \in \cs(V)$ and $0 \le r < r_0$ we
        define a (possibly infinite) seminorm
        \begin{equation}
            \label{eq:DefSeminorms}
            \seminorm{p}_{r,\seminorm{q}}^{(R)}
            (v)
            \coloneqq
            \sum_{n=0}^{\infty}
            n!^{R-1} \cdot r^n
            \cdot
            \sup_{\xi_1,\ldots,\xi_n \in \UnitBall}
            \seminorm{q}
            \bigl(
                \xi_1 \cdots \xi_n
                \acts
                v
            \bigr)
            \qquad
            \textrm{for all }
            v \in V.
        \end{equation}

    \item The subspace
        \begin{equation}
            \label{eq:AnalyticSeminorms}
            \Analytic_{R,r_0}(\varrho)
            \coloneqq
            \bigl\{
                v \in V
                \;\big|\;
                \forall_{\seminorm{q} \in \cs(V), \; 0 \le r < r_0}
                \colon
                \seminorm{p}_{r,\seminorm{q}}^{(R)}
                (v)
                <
                \infty
            \bigr\}
        \end{equation}
        is called the space of analytic vectors for $\varrho$ of order
        $R$ and with radius of convergence at least~$r_0$. We endow it
        with the locally convex topology induced by the system of
        seminorms
        \begin{equation}
            \label{eq:CollectionOfSeminorms}
            \bigl\{
                \seminorm{p}_{r,\seminorm{q}}^{(R)}
             \bigr\}_{\seminorm{q} \in \cs(V), \; 0 \le r < r_0}.
        \end{equation}
    \end{definitionlist}
\end{definition}

\begin{remark}
    \label{rem:AnalyticVectors}%
    Let $R \ge 0$, $r_0 > 0$ and
    $\varrho \colon \liealg{g} \longrightarrow \Linear(V)$ a Lie
    algebra representation.
    \begin{remarklist}
    \item It is sometimes useful to replace the open ball within
        \eqref{eq:AnalyticSeminorms} with the closed ball. As the Lie
        algebra representation $\varrho$ acts by linear and thus
        homogeneous mappings, this does not change the resulting
        seminorms.

    \item \label{item:AnalyticSeminormsMonotonicity}%
        If $0 \le r \le s$ and $R \le S$, then
        \begin{equation}
            \label{eq:SeminormsBasicEstimates}
            \seminorm{p}_{r,\seminorm{q}}^{(R)}
            \le
            \seminorm{p}_{s,\seminorm{q}}^{(R)}
            \qquad \textrm{and} \qquad
            \seminorm{p}_{r,\seminorm{q}}^{(R)}
            \le
            \seminorm{p}_{r,\seminorm{q}}^{(S)}
        \end{equation}
        and thus
        $\Analytic_{R,s}(\varrho) \subseteq \Analytic_{R,r}(\varrho)$
        and
        $\Analytic_{S,r}(\varrho) \subseteq
        \Analytic_{R,r}(\varrho)$. In particular, it suffices to
        establish estimates for $r < r_0$ bounded away from zero.

    \item \label{item:AnalyticFiner}%
        Setting $r = 0$ in \eqref{eq:AnalyticSeminorms} yields
        \begin{equation}
            \label{eq:SeminormsFiner}
            \seminorm{p}_{0,\seminorm{q}}^{(R)}
            =
            \seminorm{q}
            \qquad
            \textrm{for all }
            \seminorm{q}
            \in
            \cs(V).
        \end{equation}
        By \ref{item:AnalyticSeminormsMonotonicity}, this means that
        the $\Analytic_{R,r}(\varrho)$-topology is finer than the
        locally convex topology induced by the inclusion
        $\Analytic_{R,r_0}(\varrho) \subseteq V$ for all $r_0 \ge
        0$. In particular, if $V$ is Hausdorff, then so is
        $\Analytic_{R,r_0}(\varrho)$.

    \item If $V$ is first countable, then so is $\Analytic_{R,r}(\varrho)$, as it suffices to
    consider the radii
    \begin{equation}
        r_n
        \coloneqq
        r_0
        -
        \frac{1}{n}
    \end{equation}
    for all $n \in \N$ such that $r_n > 0$.

    \item \label{item:InfiniteDimensional}%
        There exists a version of
        Definition~\ref{def:AnalyticVectors} admitting for
        infinite-dimensional Lie algebras $\liealg{g}$. If
        $\liealg{g}$ is normed, no changes are necessary. For a
        locally convex Lie algebra, one first absorbs the radius into
        the supremum, which yields
        \begin{equation}
            \label{eq:RescalBallsInSeminorm}
            \seminorm{p}_{r,\seminorm{q}}^{(R)}
            =
            \sum_{n=0}^{\infty}
            n!^{R-1}
            \cdot
            \sup_{\xi_1,\ldots,\xi_n \in \Ball_r(0)}
            \seminorm{q}
            \bigl(
                \xi_1 \cdots \xi_n
                \acts
                v
            \bigr)
        \end{equation}
        for all $\seminorm{q} \in \cs(V)$, and then replaces
        $\Ball_r(0)$ with suitably chosen subsets of $V$. Natural
        demands then include compactness, boundedness and variations thereof.
        We leave this vast territory for future investigations.
    \end{remarklist}
\end{remark}

If $V$ is finite-dimensional, then every vector is analytic of order
$R = 1$ and with positive radius of convergence. Such vectors are
precisely the ones we deform within Section~\ref{subsec:ExamplesAxPlusB} and
Section~\ref{subsec:ExamplesHeisenberg}.
\begin{lemma}
    \label{lem:AnalyticFiniteDimensional}%
    Let $\varrho \colon \liealg{g} \longrightarrow \Linear(V)$ be a
    Lie algebra representation on a finite-dimensional vector space
    $V$, endowed with its canonical norm topology. Then, for every $v \in V$ there
    exists some $r_v > 0$ such that $v \in \Analytic_{1,r_v}(\varrho)$.
\end{lemma}
\begin{proof}
    Let $v \in V$, choose an inner product on $\liealg{g}$ and let
    $(\basis{e}_1, \ldots, \basis{e}_d) \subseteq \liealg{g}$ be an
    orthonormal basis. As
    \begin{equation*}
        \varrho(\basis{e}_j)
        \colon
        V \longrightarrow V
    \end{equation*}
    are endomorphisms of a finite dimensional space, they are
    continuous. Let
    \begin{equation*}
        C
        \coloneqq
        \max_{j=1,\ldots,d}
        \norm[\big]
        {
            \varrho(e_j)
        },
    \end{equation*}
    where $\norm{\argument}$ denotes the operator norm. This yields
    the estimate
    \begin{align*}
        \seminorm{p}_{r,\norm{\argument}}^{(1)}(v)
        &=
        \sum_{n=0}^{\infty}
        r^n
        \sup_{\xi_1, \ldots, \xi_n \in \UnitBall}
        \norm[\big]
        {
            \xi_1 \cdots \xi_n
            \acts
            v
        } \\
        &\le
        \sum_{n=0}^{\infty}
        r^n
        \sup_{\xi_1, \ldots, \xi_n \in \UnitBall}
        \abs{\xi_1}^{j_1}
        \cdots
        \abs{\xi_n}^{j_n}
        \cdot
        \norm[\big]{\varrho(\basis{e}_{j_1})}
        \cdots
        \norm[\big]{\varrho(\basis{e}_{j_n})}
        \cdot
        \norm{v} \\
        &\le
        \norm{v}
        \cdot
        \sum_{n=0}^{\infty}
        (C \cdot r \cdot d)^n \\
        &=
        \frac{\norm{v}}{1 - C \cdot r \cdot d}
    \end{align*}
    for all $v \in V$ and $r < r_v \coloneqq 1/(C \cdot d)$, where we
    use that the basis coefficients of elements in the unit ball of a
    Hilbert space lie in the unit disk.  Hence,
    $v \in \Analytic_{1,r_v}(\varrho)$.
\end{proof}

The analytic vectors inherit completeness properties from $V$. Recall
that a Hausdorff locally convex space is called quasi-complete if
every bounded Cauchy net is convergent.
\begin{proposition}
    \label{prop:AnalyticCompleteness}%
    Let $\varrho \colon \liealg{g} \longrightarrow \Linear(V)$ be a
    Lie algebra representation on a (quasi, sequentially) complete
    Hausdorff locally convex space~$V$. Then
    $\Analytic_{R,r_0}(\varrho)$ is (quasi, sequentially) complete
    Hausdorff for all $R \ge 0$ and $r_0 > 0$.
\end{proposition}
\begin{proof}
    We have already discussed the Hausdorff property. Let
    $(v_\alpha)_{\alpha \in J} \subseteq \Analytic_{R,r_0}(\varrho)$
    be a Cauchy net. By Remark~\ref{rem:AnalyticVectors},
    \ref{item:AnalyticFiner}, the
    $\Analytic_{R,r_0}(\varrho)$-topology is finer than the subspace
    topology induced by the inclusion
    $\Analytic_{R,r_0}(\varrho) \subseteq V$. Thus, as $V$ is complete
    Hausdorff by assumption, there exists some vector~$v \in V$ such
    that $v_\alpha \rightarrow v$ within the topology of $V$. Let
    $\epsilon > 0$, $\seminorm{q} \in \cs(V)$ and $0 < r <
    r_0$. Unwrapping the Cauchy condition yields some
    index~$\alpha_0 \in J$ with
    \begin{equation*}
        \seminorm{p}_{r,\seminorm{q}}^{(R)}
        \bigl(
            v_\alpha - v_\beta
        \bigr)
        =
        \sum_{n=0}^{\infty}
        n!^{R-1}
        \cdot
        r^n
        \sup_{\xi_1, \ldots, \xi_{n} \in \UnitBall}
        \seminorm{q}
        \bigl(
            \xi_1 \cdots \xi_n
            \acts
            (v_\alpha - v_\beta)
        \bigr)
        \le
        \epsilon
    \end{equation*}
    for all $\alpha, \beta \later \alpha_0$. In particular,
    \begin{equation*}
        \sum_{n=0}^{N}
        n!^{R-1}
        \cdot
        r^n
        \cdot
        \seminorm{q}
        \bigl(
            \xi_1^{(n)} \cdots \xi_n^{(n)}
            \acts
            (v_\alpha - v_\beta)
        \bigr)
        \le
        \epsilon
    \end{equation*}
    for all $N \in \N_0$,
    $\xi_1^{(1)}, \xi_1^{(2)}, \xi_2^{(2)}, \ldots, \xi_N^{(N)} \in
    \UnitBall$ and $\alpha,\beta \later \alpha_0$, as all terms are
    non-negative. By continuity of $\varrho(\xi)$ for all
    $\xi \in \liealg{g}$ and the convergence $v_\alpha \rightarrow v$
    within~$V$, this in turn implies
    \begin{equation*}
        \sum_{n=0}^{N}
        n!^{R-1}
        \cdot
        r^n
        \cdot
        \seminorm{q}
        \bigl(
        \xi_1^{(n)} \cdots \xi_n^{(n)}
        \acts
            (v_\alpha - v)
        \bigr)
        \le
        \epsilon
    \end{equation*}
    for all $N \in \N_0$,
    $\xi_1^{(1)}, \xi_1^{(2)}, \xi_2^{(2)}, \ldots, \xi_N^{(N)} \in
    \UnitBall$ and $\alpha \later \alpha_0$. Reintroducing the
    suprema leads to
    \begin{equation*}
        \sum_{n=0}^{N}
        n!^{R-1}
        \cdot
        r^n
        \cdot
        \sup_{\xi_1, \ldots, \xi_{n} \in \UnitBall}
        \seminorm{q}
        \bigl(
            \xi_1 \cdots \xi_n
            \acts
            (v_\alpha - v)
        \bigr)
        \le
        \epsilon
    \end{equation*}
    for all $N \in \N_0$ and $\alpha \later \alpha_0$. Hence, the
    limit $N \rightarrow \infty$ exists by monotonicity, and we get
    \begin{equation*}
        \seminorm{p}_{r,\seminorm{q}}^{(R)}
        \bigl(
            v_\alpha - v
        \bigr)
        =
        \sum_{n=0}^{\infty}
        n!^{R-1}
        \cdot
        r^n
        \cdot
        \sup_{\xi_1, \ldots, \xi_{n} \in \UnitBall}
        \seminorm{q}
        \bigl(
            \xi_1 \cdots \xi_n
            \acts
            (v_\alpha - v)
        \bigr)
        \le
        \epsilon
    \end{equation*}
    for all $\alpha \later \alpha_0$, proving $v_\alpha \rightarrow v$
    within the $\Analytic_{R,r_0}(\varrho)$-topology. Finally, the
    triangle inequality yields
    \begin{equation*}
        \seminorm{p}_{r,\seminorm{q}}^{(R)}
        (v)
        \le
        \seminorm{p}_{r,\seminorm{q}}^{(R)}
        \bigl(
            v_{\alpha_0} - v
        \bigr)
        +
        \seminorm{p}_{r,\seminorm{q}}^{(R)}
        \bigl(
            v_{\alpha_0}
        \bigr)
        \le
        \epsilon
        +
        \seminorm{p}_{r,\seminorm{q}}^{(R)}
        \bigl(
            v_{\alpha_0}
        \bigr)
        <
        \infty.
    \end{equation*}
    Hence, $v \in \Analytic_{R,r}(\varrho)$, completing the proof of
    completeness.  Mutatis mutandis, the arguments for sequential and
    quasi-completeness are identical. For the latter, the only
    additional ingredient is Remark~\ref{rem:AnalyticVectors},
    \ref{item:AnalyticFiner}: Any bounded subset of
    $\Analytic_{R,r}(\varrho)$ is also bounded within $V$.
\end{proof}

Next, we establish the $\varrho$-invariance of the spaces of analytic
vectors.
\begin{proposition}
    \label{prop:AnalyticInvariance}%
    Let $\varrho \colon \liealg{g} \longrightarrow \Linear(V)$ be a
    Lie algebra representation and $R \ge 0$. Then $\varrho$ restricts
    to representations
    \begin{equation}
        \label{eq:AnalyticInvarianceConstituents}
        \varrho_{r_0}
        \colon
        \liealg{g}
        \longrightarrow
        \Linear
        \bigl(
            \Analytic_{R,r_0}(\varrho)
        \bigr)
        \qquad
        \textrm{for all }
        r_0 > 0.
    \end{equation}
\end{proposition}
\begin{proof}
    Let $\xi \in \liealg{g}$ and $v \in \Analytic_{R,r_0}(\varrho)$.
    The idea is to view $\seminorm{p}_{r,\seminorm{q}}^{(R)}$
    as a power series with respect to the radius $r$ on the disk around
    the origin with radius $r_0$. In particular, we may plug in arbitrary
    $z \in \field{C}$ with $\abs{z} < r_0$. Fix
    $\seminorm{q} \in \cs(V)$.  We estimate
    \begin{align*}
        \seminorm{p}_{r,\seminorm{q}}^{(R)}
        \bigl(
        \xi \acts v
        \bigr)
        &=
        \sum_{n = 0}^\infty
        n!^{R-1}
        \cdot
        r^n
        \cdot
        \sup_{\xi_1, \ldots, \xi_n \in \UnitBall}
        \seminorm{q}
        \bigl(
            \xi_1 \cdots \xi_n \xi
            \acts
            v
        \bigr) \\
        &\le
        \norm{\xi}
        \cdot
        \sum_{n = 0}^\infty
        n!^{R-1}
        \cdot
        r^n
        \cdot
        \sup_{\xi_1, \ldots, \xi_{n+1} \in \UnitBall}
        \seminorm{q}
        \bigl(
            \xi_1 \cdots \xi_{n+1}
            \acts
            v
        \bigr) \\
        &=
        \norm{\xi}
        \cdot
        \sum_{n = 1}^\infty
        (n-1)!^{R-1}
        \cdot
        r^{n-1}
        \cdot
        \sup_{\xi_1, \ldots, \xi_{n} \in \UnitBall}
        \seminorm{q}
        \bigl(
            \xi_1 \cdots \xi_{n}
            \acts
            v
        \bigr) \\
        &=
        \norm{\xi}
        \cdot
        \sum_{n = 1}^\infty
        n!^{R-1}
        \cdot
        n^{-R}
        \cdot
        \frac{\D r^n}{\D r}
        \cdot
        \sup_{\xi_1, \ldots, \xi_{n} \in \UnitBall}
        \seminorm{q}
        \bigl(
            \xi_1 \cdots \xi_{n}
            \acts
            v
        \bigr) \\
        &\le
        \norm{\xi}
        \cdot
        \frac{\D}{\D r}
        \sum_{n = 0}^\infty
        n!^{R-1}
        \cdot
        r^n
        \cdot
        \sup_{\xi_1, \ldots, \xi_{n} \in \UnitBall}
        \seminorm{q}
            \bigl(
            \xi_1 \cdots \xi_{n}
            \acts
            v
        \bigr) \\
        &=
        \norm{\xi}
        \cdot
        \frac{\D}{\D r}
        \seminorm{p}_{r,\seminorm{q}}^{(R)}(v)
    \end{align*}
    for all $0 \le r < r_0$, which also holds in the degenerate case
    $\xi = 0$. As the differentiation operator is continuous on the
    space of holomorphic functions $\Holomorphic(\disk_{r_0}(0))$ on
    the open disk with radius $r_0$, there exist $C_r > 0$ and
    $0 \le s_r < r_0$ only depending on $0 \le r < r_0$ such that
    \begin{equation*}
        \max_{\abs{z} \le r}
        \abs[\big]
        {
            f'(z)
        }
        \le
        C_r
        \cdot
        \max_{\abs{z} \le s_r}
        \abs[\big]
        {f(z)}
        \qquad
        \textrm{for all }
        f
        \in
        \Holomorphic
        \bigl(
            \disk_{r_0}(0)
        \bigr).
    \end{equation*}
    As all coefficients of the power series
    $z \mapsto \seminorm{p}_{z,\seminorm{q}}^{(R)}(v)$ are
    non-negative, this implies
    \begin{equation*}
        \frac{\D}{\D r}
        \seminorm{p}_{r,\seminorm{q}}^{(R)}(v)
        \le
        \max_{\abs{z} \le r}
        \abs[\bigg]
        {
            \frac{\D}{\D z}
            \seminorm{p}_{z,\seminorm{q}}^{(R)}(v)
        }
        \le
        C_r
        \cdot
        \max_{\abs{z} \le s_r}
        \abs[\big]
        {\seminorm{p}_{z,\seminorm{q}}^{(R)}(v)}
        =
        C_r
        \cdot
        \seminorm{p}_{s_r,\seminorm{q}}^{(R)}(v)
    \end{equation*}
    for all $0 \le r < r_0$ and $\seminorm{q} \in \cs(V)$. Plugging
    this into our earlier considerations yields
    \begin{equation*}
        \xi \acts v \in \Analytic_{R,r_0}(\varrho)
    \end{equation*}
    and the continuity estimate
    \begin{equation*}
        \seminorm{p}_{r,\seminorm{q}}^{(R)}
        \bigl(
            \xi \acts v
        \bigr)
        \le
        C_r
        \cdot
        \norm{\xi}
        \cdot
        \seminorm{p}_{s_r,\seminorm{q}}^{(R)}(v)
        \qquad
        \textrm{for all }
        v \in \Analytic_{R,r_0}(\varrho)
    \end{equation*}
    for the linear mapping
    \begin{equation*}
        \varrho(\xi)
        \colon
        \Analytic_{R,r_0}(\varrho)
        \longrightarrow
        \Analytic_{R,r_0}(\varrho).
    \end{equation*}
    This establishes the desired $\varrho$-invariance of $\Analytic_{R,r_0}(\varrho)$.
\end{proof}

Slightly adapting our argument results in the following more explicit
inequalities, which may be viewed as an abstract version of the Cauchy
estimates for holomorphic functions.
\begin{lemma}[Cauchy estimates]
    \label{lem:CauchyEstimates}%
    Let $\varrho \colon \liealg{g} \longrightarrow \Linear(V)$ be a
    Lie algebra representation. Moreover, let $r_0 > 0$,
    $\seminorm{q} \in \cs(V)$ and $R \ge 0$.  Then
    \begin{equation}
        \label{eq:CauchyEstimates}
        \seminorm{p}_{r,\seminorm{q}}^{(R)}
        \bigl(
        \varrho(\xi_1 \cdots \xi_n)v
        \bigr)
        \le
        \frac{n!^{1-R}}{r^n}
        \cdot
        \seminorm{p}_{2r,\seminorm{q}}^{(R)}
        (v)
    \end{equation}
    for all $0 < 2r < r_0$, $n \in \field{N}_0$,
    $\xi_1, \ldots, \xi_n \in \UnitBall^\cl$ and
    $v \in \Analytic_{R,r_0}(\varrho)$.
\end{lemma}
\begin{proof}
    Similar to before, we have
    \begin{align*}
        \seminorm{p}_{r,\seminorm{q}}^{(R)}
        \bigl(
        \varrho(\xi_1 \cdots \xi_n)v
        \bigr)
        &=
        \sum_{k=0}^{\infty}
        k!^{R-1} \cdot r^k
        \sup_{\eta_1,\ldots,\eta_k \in \UnitBall}
        \seminorm{q}
        \bigl(
        \eta_1 \cdots \eta_k
        \xi_1 \cdots \xi_n
        \acts
        v
        \bigr) \\
        &\le
        \sum_{k=0}^{\infty}
        k!^{R-1} \cdot r^k
        \sup_{\eta_1,\ldots,\eta_{k+n} \in \UnitBall}
        \seminorm{q}
        \bigl(
        \eta_1 \cdots \eta_{k+n}
        \acts
        v
        \bigr) \\
        &=
        \frac{n!^{1-R}}{r^n}
        \sum_{k=0}^{\infty}
        \binom{k+n}{k}^{1-R}
        (k+n)!^{R-1} \cdot r^{k+n}
        \sup_{\eta_1,\ldots,\eta_{k+n} \in \UnitBall}
        \seminorm{q}
        \bigl(
        \eta_1 \cdots \eta_{k+n}
        \acts
        v
        \bigr) \\
        &\le
        \frac{n!^{1-R}}{r^n}
        \sum_{k=0}^{\infty}
        (k+n)!^{R-1} \cdot (2r)^{k+n}
        \sup_{\eta_1,\ldots,\eta_{k+n} \in \UnitBall}
        \seminorm{q}
        \bigl(
        \eta_1 \cdots \eta_{k+n}
        \acts
        v
        \bigr) \\
        &=
        \frac{n!^{1-R}}{r^n}
        \sum_{k=n}^{\infty}
        k!^{R-1} \cdot (2r)^k
        \sup_{\eta_1,\ldots,\eta_{k} \in \UnitBall}
        \seminorm{q}
        \bigl(
        \eta_1 \cdots \eta_{k}
        \acts
        v
        \bigr) \\
        &\le
        \frac{n!^{1-R}}{r^n}
        \cdot
        \seminorm{p}_{2r,\seminorm{q}}^{(R)}
        (v),
    \end{align*}
    where we have employed the estimate
    \begin{equation*}
        \binom{k+n}{k}^{1-R}
        \le
        \binom{k+n}{k}
        \le
        2^{k+n}
        \qquad
        \textrm{for all }
        k \in \N_0.
    \end{equation*}
\end{proof}

To conclude our discussion of analytic vectors, we show that
the passage to analytic vectors may be regarded as a functor
$\Analytic_{R,r}$:
\begin{lemma}
    \label{lem:AnalyticFunctoriality}%
    Let
    \begin{equation}
        \varrho
        \colon
        \liealg{g}
        \longrightarrow
        \Linear(V)
        \qquad \textrm{and} \qquad
        \tilde{\varrho}
        \colon
        \tilde{\liealg{g}}
        \longrightarrow
        \Linear(\tilde{V})
    \end{equation}
    be Lie algebra representations with a continuous intertwiner
    $L \in \Linear(V,W)$ along a surjective Lie algebra morphism
    $\phi \colon \liealg{g} \longrightarrow \tilde{\liealg{g}}$, i.e.
    \begin{equation}
        L
        \bigl(
        \xi \acts v
        \bigr)
        =
        \phi(\xi)
        \acts
        Lv
        \qquad
        \textrm{for all }
        \xi \in \liealg{g},
        v \in V.
    \end{equation}
    Then, for every $r_0 > 0$, there exists an $s_0 > 0$ such that $L$
    restricts to a continuous linear mapping
    \begin{equation}
        \label{eq:IntertwinerRestriction}
        L
        \colon
        \Analytic_{R,r_0}(\varrho)
        \longrightarrow
        \Analytic_{R,s_0}(\tilde{\varrho})
        \qquad
        \textrm{for all }
        R \ge 0.
    \end{equation}
    If $\phi$ is a surjective contraction, i.e. $\norm{\phi} \le 1$,
    then we may choose $s_0 = r_0$.
\end{lemma}
\begin{proof}
    Let $v \in \Analytic_{R,r_0}(\varrho)$ and
    $\tilde{\seminorm{q}} \in \cs(\tilde{V})$.  By continuity of $L$,
    there exists a corresponding continuous seminorm
    $\seminorm{q} \in \cs(V)$ such that
    \begin{equation*}
        \tilde{\seminorm{q}}
        \bigl(
            Lx
        \bigr)
        \le
        \seminorm{q}(x)
        \qquad
        \textrm{for all }
        x \in V.
    \end{equation*}
    As $\phi$ is a linear surjection and thus open, we moreover find
    some $r > 0$ such that
    \begin{equation*}
        \tilde{\UnitBall}
        \subseteq
        r
        \cdot
        \phi
        \bigl(
            \UnitBall
        \bigr)
        =
        \phi
        \bigl(
            \Ball_{r}(0)
        \bigr)
        \subseteq
        \tilde{\liealg{g}},
    \end{equation*}
    where we denote the open unit ball of $\tilde{\liealg{g}}$ by
    $\tilde{\UnitBall}$.  Using the equivariance, this implies
    \begin{align*}
        \seminorm{p}_{s,\tilde{\seminorm{q}}}^{(R)}
        \bigl(
            Lv
        \bigr)
        &=
        \sum_{n=0}^{\infty}
        n!^{R-1} \cdot s^n
        \sup_{\tilde\xi_1, \ldots, \tilde\xi_{n} \in \tilde{\UnitBall} \subseteq
        \tilde{\liealg{g}}}
        \tilde{\seminorm{q}}
        \bigl(
            \tilde \xi_1 \cdots \tilde \xi_n
            \acts
            Lv
        \bigr) \\
        &\le
        \sum_{n=0}^{\infty}
        n!^{R-1} \cdot s^n
        \sup_{\eta_1, \ldots, \eta_{n} \in \Ball_{r}(0) \subseteq \liealg{g}}
        \tilde{\seminorm{q}}
        \bigl(
            \phi(\eta_1) \cdots \phi(\eta_n)
            \acts
            Lv
        \bigr) \\
        &=
        \sum_{n=0}^{\infty}
        n!^{R-1} \cdot (r \cdot s)^n
        \sup_{\eta_1, \ldots, \eta_{n} \in \UnitBall \subseteq \liealg{g}}
        \tilde{\seminorm{q}}
        \bigl(
            L
            (
                \eta_1 \cdots \eta_n
                \acts
                v
        )
        \bigr) \\
        &\le
        \sum_{n=0}^{\infty}
        n!^{R-1} \cdot (r \cdot s)^n
        \sup_{\eta_1, \ldots, \eta_{n} \in \UnitBall \subseteq \liealg{g}}
        \seminorm{q}
        \bigl(
            \eta_1 \cdots \eta_n
            \acts
            v
        \bigr) \\
        &=
        \seminorm{p}_{r \cdot s, \seminorm{q}}(v)
    \end{align*}
    for all $s < s_0 \coloneqq r_0/r$. Hence,
    $v \in \Analytic_{R,s_0}(\tilde{\varrho})$ and variation of $v$
    proves the continuity of the restricted
    intertwiner \eqref{eq:IntertwinerRestriction}.
\end{proof}

Taking Lie algebra representations as objects and continuous
intertwiners along contractive surjective Lie algebra morphisms as
morphisms results in the Lie algebraic representation category
$\categoryname{Rep}$. This allows us to rephrase
Lemma~\ref{lem:AnalyticFunctoriality} in the following manner.
\begin{corollary}
    \label{cor:AnalyticFunctoriality}%
    The assignment
    \begin{equation}
        \Analytic_{R,r}
        \colon
        \categoryname{Rep}
        \longrightarrow
        \categoryname{Rep}, \quad
        \bigl(
            \varrho
            \colon
            \liealg{g}
            \longrightarrow
            \Linear(V)
        \bigr)
        \mapsto
        \bigl(
            \varrho
            \colon
            \liealg{g}
            \longrightarrow
            \Linear(\Analytic_{R,r}(\varrho))
        \bigr)
    \end{equation}
    acting as restriction on the intertwiners constitutes a covariant
    functor.
\end{corollary}

Finally, note that the functor $\Analytic_{R,r}$ restricts properly if
we additionally impose the Hausdorff property on the representation
spaces.

%% file: TeX/PreliminariesProjectiveTensors.tex

As the mappings we deform are bilinear, continuity considerations
automatically lead to topological tensor products. Throughout the
text, we shall use projective tensor products, as they constitute the
tensor product in the category of locally convex spaces.

Let $V$ and $W$ be locally convex spaces. The first crucial property
of projective tensor products of seminorms $\seminorm{q} \in \cs(V)$
and $\seminorm{p} \in \cs(W)$ we shall frequently use is
\begin{equation}
    \label{eq:ProjectiveOnFactorizing}
    \bigl(
        \seminorm{q}
        \tensor
        \seminorm{p}
    \bigr)
    (v \tensor w)
    =
    \seminorm{q}(v)
    \cdot
    \seminorm{p}(w)
    \qquad
    \textrm{for all }
    v \in V, w \in W.
\end{equation}
Unwrapping the universal property leads to what is known as the
infimum argument, see e.g.~\cite[Prop.~43.4]{treves:2006a}.
\begin{proposition}[Infimum argument]
    \label{prop:InfimumArgument}%
    Let $V_1, \ldots, V_n, W$ be locally convex spaces and
    \begin{equation}
        \phi
        \colon
        V_1 \times \cdots \times V_n
        \longrightarrow
        W
    \end{equation}
    be $n$-linear with corresponding linear map
    \begin{equation}
        \Phi
        \colon
        V_1 \tensor \cdots \tensor V_n
        \longrightarrow
        W.
    \end{equation}
    We endow $V_1 \times \cdots \times V_n$ with the Cartesian product
    topology and $V_1 \tensor \cdots \tensor V_n$ with the projective
    tensor product topology.  Then $\phi$ is continuous if and only if
    $\Phi$ is. More precisely, if for a continuous seminorm
    $\seminorm{q} \in \cs(W)$ there are $\seminorm{p}_1 \in \cs(V_1)$,
    $\ldots$, $\seminorm{p_n} \in \cs(V_n)$ such that
    \begin{equation}
        \label{eq:ContinuityMultilinear}
        \seminorm{q}
        \bigl(
        \phi
        (v_1, \ldots, v_n)
        \bigr)
        \le
        \seminorm{p}_1(v_1)
        \cdots
        \seminorm{p}_n(v_n)
        \qquad
        \textrm{for all }
        v_1 \in V_1,
        \ldots,
        v_n \in V_n,
    \end{equation}
    then
    \begin{equation}
        \seminorm{q}
        \bigl(
        \Phi(v)
        \bigr)
        \le
        \bigl(
        \seminorm{p}_1
        \tensor \cdots \tensor
        \seminorm{p}_n
        \bigr)(v)
        \qquad
        \textrm{for all }
        v \in V_1 \tensor \cdots \tensor V_n,
    \end{equation}
    and vice versa.
\end{proposition}

As this is the only type of topological tensor product we shall need,
we simply denote it by $\tensor$ and suppress the usual decorative
indices for the sake of readability. A systematic treatment of
projective tensor products can be found in \cite[§41]{koethe:1979a}.

%% file: TeX/ConvergenceOfUDFEntire.tex

Among the analytic vectors of order $R$, we have the particularly well
behaved \emph{entire vectors}, which were conceived and studied
systematically within \cite{goodman:1969a}, though only for
order~$R=0$ and for representations on Banach spaces.
\begin{definition}[Entire Vectors]
    Let $\varrho \colon \liealg{g} \longrightarrow \Linear(V)$ be a
    Lie algebra representation. The space of entire vectors for
    $\varrho$ of order $R \ge 0$ is defined as the projective limit
    \begin{equation}
        \label{eq:EntireVectors}
        \Entire_R(\varrho)
        \coloneqq
        \varprojlim_{r > 0}
        \Analytic_{R,r}(\varrho)
        \cong
        \bigcap_{r > 0}
        \Analytic_{R,r}(\varrho).
    \end{equation}
\end{definition}

Recall that the projective limit topology is simply the initial
topology with respect to the canonical inclusions
\begin{equation}
    \Entire_R(\varrho)
    \overset{\iota_{r}}{\hookrightarrow}
    \Analytic_{R,r}(\varrho)
    \qquad
    \textrm{for all }
    r > 0.
\end{equation}
In particular, this means that
\begin{equation}
    \label{eq:EntireSeminorms}
    \bigl\{
        \seminorm{p}_{r,\seminorm{q}}
        \colon
        r \ge 0,
        \seminorm{q} \in \cs(V)
    \bigr\}
\end{equation}
constitutes a defining system of seminorms for
$\Entire_R(\varrho)$. This, in turn, implies the inheritance of the
Hausdorff property, first countability and completeness for
$\Entire_R(\varrho)$ from $V$ analogous to
Proposition~\ref{prop:AnalyticCompleteness}. A systematic treatment of
projective limits of locally convex spaces can be found in
\cite[Sec.~II.5]{schaefer:1999a}.
\begin{corollary}
    \label{cor:EntireCompleteness}%
    Let $\varrho \colon \liealg{g} \longrightarrow \Linear(V)$ be a
    Lie algebra representation and $R \ge 0$.
    \begin{corollarylist}
    \item If $V$ is (quasi, sequentially) complete Hausdorff, then so
        is $\Entire_R(\varrho)$.
    \item If $V$ is a Fréchet space, then so is $\Entire_R(\varrho)$.
    \end{corollarylist}
\end{corollary}

The first step towards universal deformation formulas for entire
vectors is now that passing to entire vectors preserves continuous
malleable $\liealg{g}$-triples.
\begin{proposition}
    \label{prop:EntireOfMalleable}%
    Let
    \begin{equation}
        \label{eq:EntireOfMalleable}
        \mu
        \colon
        V \tensor W \longrightarrow X
    \end{equation}
    be a continuous malleable $\liealg{g}$-triple. Then so is
    \begin{equation}
        \label{eq:EntireOfMalleableTriple}
        \mu
        \colon
        \Entire_R(\varrho_V)
        \tensor
        \Entire_R(\varrho_W)
        \longrightarrow
        \Entire_R(\varrho_X)
    \end{equation}
    for all $R \ge 0$. More precisely, we have the continuity estimate
    \begin{equation}
        \label{eq:ContinuityMultiplicationMalleable}
        \seminorm{p}_{r, \seminorm{q}_X}^{(R)}
        \bigl(
            \mu(z)
        \bigr)
        \le
        \bigl(
            \seminorm{p}_{2^R r, \seminorm{q}_V}^{(R)}
            \tensor
            \seminorm{p}_{2^R r, \seminorm{q}_W}^{(R)}
        \bigr)
        (z)
    \end{equation}
    for all $r \ge 0$,
    $z \in \Entire_R(\varrho_V) \tensor \Entire_R(\varrho_W)$ and
    seminorms $\seminorm{q}_V \in \cs(V)$, $\seminorm{q}_W \in \cs(W)$
    and $\seminorm{q}_X \in \cs(X)$ such that
    \begin{equation}
        \label{eq:ContinuityMultiplication}
        \seminorm{q}_X
        \bigl(
            \mu(z)
        \bigr)
        \le
        \bigl(
            \seminorm{q}_V
            \tensor
            \seminorm{q}_W
        \bigr)(z)
        \qquad
        \textrm{for all }
        z \in V \tensor W.
    \end{equation}
\end{proposition}
\begin{proof}
    Let $Z \in \{V,W,X\}$. By
    Proposition~\ref{prop:AnalyticInvariance}, the mapping $\varrho_Z$
    restricts to a Lie algebra representation
    on~$\Analytic_{R,r_0}(\varrho_Z)$ for all $r_0 > 0$. Given
    $\xi \in \liealg{g}$, and in view of \eqref{eq:EntireVectors}, we
    thus get a linear mapping
    \begin{equation*}
        \varrho(\xi)
        \at[\Big]{\Entire_R(\varrho_Z)}^{\Entire_R(\varrho_Z)}
        \colon
        \Entire_R(\varrho_Z)
        \longrightarrow
        \Entire_R(\varrho_Z).
    \end{equation*}
    By the characteristic property of the initial topology, its
    continuity is equivalent to the continuity of
    \begin{equation*}
        \iota_r
        \circ
        \varrho(\xi)
        \at[\Big]{\Entire_R(\varrho_Z)}^{\Entire_R(\varrho_Z)}
        =
        \varrho(\xi)
        \circ
        \iota_r
        \colon
        \Entire_R(\varrho_Z)
        \longrightarrow
        \Analytic_{R,r}(\varrho_Z)
        \qquad
        \textrm{for all }
        r > 0,
    \end{equation*}
    which we have already established. Hence, $\varrho_Z$ restricts to
    a Lie algebra representation on~$\Entire_R(\varrho_Z)$. It remains
    to check that \eqref{eq:EntireOfMalleable} maps
    $\Entire_R(\varrho_V) \tensor \Entire_R(\varrho_W)$ into
    $\Entire_R(\varrho_X)$ in a continuous fashion. To show this, let
    $r > 0$, $v \in \Entire_R(\varrho_V)$,
    $w \in \Entire_R(\varrho_W)$ and continuous seminorms
    $\seminorm{q}_V \in \cs(V)$, $\seminorm{q}_W \in \cs(W)$ and
    $\seminorm{q}_X \in \cs(X)$ such that the continuity
    estimate~\eqref{eq:ContinuityMultiplication} for $\mu$ holds.
    Combining the latter with \eqref{eq:MalleableSystemLeibniz},
    we first note
    \begin{align*}
        &\sup_{\xi_1, \ldots, \xi_n \in \UnitBall}
        \seminorm{q}_X
        \bigl(
            \xi_1 \cdots \xi_n
            \acts
            \mu(v \tensor w)
        \bigr) \\
        &\le
        \sup_{\xi_1, \ldots, \xi_n \in \UnitBall}
        \sum_{k=0}^n
        \sum_{\sigma \in \Shuffle(n,n-k)}
        \seminorm{q}_X
        \Bigl(
            \mu
            \bigl(
                (
                    \xi_{\sigma(1)} \cdots \xi_{\sigma(k)}
                    \acts
                    v
                )
                \tensor
                (
                    \xi_{\sigma(k+1)} \cdots \xi_{\sigma(n)}
                    \acts
                    w
                )
            \bigr)
        \Bigr) \\
        &\le
        \sup_{\xi_1, \ldots, \xi_n \in \UnitBall}
        \sum_{k=0}^n
        \sum_{\sigma \in \Shuffle(n,n-k)}
        \seminorm{q}_V
        \bigl(
            \xi_{\sigma(1)} \cdots \xi_{\sigma(k)}
            \acts
            v
        \bigr)
        \cdot
        \seminorm{q}_W
        \bigl(
            \xi_{\sigma(k+1)} \cdots \xi_{\sigma(n)}
            \acts
            w
        \bigr) \\
        &\le
        \sum_{k=0}^{n}
        \binom{n}{k}
        \sup_{\xi_1, \ldots, \xi_k \in \UnitBall}
        \seminorm{q}_V
        (\xi_1 \cdots \xi_k \acts v)
        \cdot
        \sup_{\eta_1, \ldots, \eta_{n-k} \in \UnitBall}
        \seminorm{q}_W
        (\eta_1 \cdots \eta_{n-k} \acts w).
    \end{align*}
    Consequently,
    \begin{align*}
        &\seminorm{p}_{r,\seminorm{q}_X}^{(R)}
        \bigl(
            \mu(v \tensor w)
        \bigr)
        =
        \sum_{n=0}^{\infty}
        n!^{R-1}
        \cdot
        r^n
        \cdot
        \sup_{\xi_1, \ldots, \xi_n \in \UnitBall}
        \seminorm{q}_X
        \bigl(
            \xi_1 \cdots \xi_n
            \acts
            \mu(v \tensor w)
        \bigr) \\
        &\le
        \sum_{n=0}^{\infty}
        n!^{R-1}
        \cdot
        r^n
        \cdot
        \sum_{k=0}^{n}
        \binom{n}{k}
        \sup_{\xi_1, \ldots, \xi_k \in \UnitBall}
        \seminorm{q}_V
        (\xi_1 \cdots \xi_k \acts v)
        \cdot
        \sup_{\eta_1, \ldots, \eta_{n-k} \in \UnitBall}
        \seminorm{q}_W
        (\eta_1 \cdots \eta_{n-k} \acts w) \\
        &=
        \sum_{n=0}^{\infty}
        \sum_{k=0}^{n}
        \binom{n}{k}^{R}
        \Bigl(
            k!^{R-1}
            \cdot
            r^k
            \cdot
            \sup_{\xi_1, \ldots, \xi_k \in \UnitBall}
            \seminorm{q}_V
            (\xi_1 \cdots \xi_k \acts v)
        \Bigr) \\
        &\hspace{5cm}
        \Bigl(
            (n-k)!^{R-1}
            \cdot
            r^{n-k}
            \cdot
            \sup_{\eta_1, \ldots, \eta_{n-k} \in \UnitBall}
            \seminorm{q}_W
            (\eta_1 \cdots \eta_{n-k} \acts w)
        \Bigr) \\
        &\le
        \sum_{n=0}^{\infty}
        \sum_{k=0}^{n}
        2^{nR}
        \cdot
        \Bigl(
            k!^{R-1}
            \cdot
            r^k
            \cdot
            \sup_{\xi_1, \ldots, \xi_k \in \UnitBall}
            \seminorm{q}_V
            (\xi_1 \cdots \xi_k \acts v)
        \Bigr) \\
        &\hspace{5cm}
        \Bigl(
            (n-k)!^{R-1}
            \cdot
            r^{n-k}
            \cdot
            \sup_{\eta_1, \ldots, \eta_{n-k} \in \UnitBall}
            \seminorm{q}_W
            (\eta_1 \cdots \eta_{n-k} \acts w)
        \Bigr) \\
        &=
        \biggl(
            \sum_{n=0}^{\infty}
            n!^{R-1}
            \cdot
            (2^R \cdot r)^n
            \cdot
            \sup_{\xi_1, \ldots, \xi_n \in \UnitBall}
            \seminorm{q}_V
            (\xi_1 \cdots \xi_k \acts v)
        \biggr) \\
        &\hspace{5cm}
        \biggl(
        \sum_{n=0}^{\infty}
            n!^{R-1}
            \cdot
            (2^R \cdot r)^n
            \sup_{\eta_1, \ldots, \eta_n \in \UnitBall}
            \seminorm{q}_W
            (\eta_1 \cdots \eta_n \acts w)
        \biggr) \\
        &=
        \seminorm{p}_{2^R r, \seminorm{q}_V}^{(R)}(v)
        \cdot
        \seminorm{p}_{2^R r, \seminorm{q}_W}^{(R)}(w),
    \end{align*}
    where we have used the estimate $\binom{n}{k} \le 2^n$ and the
    Cauchy product formula for infinite series. We have thus shown
    that $\mu(v \tensor w) \in \Entire_R(\varrho_X)$ and invoking
    Proposition~\ref{prop:InfimumArgument} establishes the continuity
    estimate \eqref{eq:ContinuityMultiplicationMalleable}.
\end{proof}
\begin{remark}
    \label{remark:AlternativeConstructionEntireR}%
    Alternatively, the continuity of $\varrho(\xi)$ is a direct
    consequence of the concrete defining system of seminorms
    \eqref{eq:EntireSeminorms} for $\Entire_{R}(\varrho)$ and the
    Cauchy estimates from Lemma~\ref{lem:CauchyEstimates}. This also
    establishes that the continuity of
    $\varrho(\xi) \colon V \longrightarrow V$ is immaterial and may
    always be enforced by passing to analytic or entire vectors.
\end{remark}

Combining Proposition~\ref{prop:EntireOfMalleable} with the
functoriality of $\Analytic_{R,r_0}$ from
Corollary~\ref{cor:AnalyticFunctoriality} results in the following
functoriality statement:
\begin{corollary}
    \label{cor:EntireFunctoriality}%
    The assignment
    \begin{equation}
        \Entire_R
        \colon
        \categoryname{cmTriple}
        \longrightarrow
        \categoryname{cmTriple}, \quad
        \Entire_R
        \bigl(
            V,W,X
        \bigr)
        \coloneqq
        \bigl(
            \Entire_R(\varrho_V),
            \Entire_R(\varrho_W),
            \Entire_R(\varrho_X)
        \bigr)
    \end{equation}
    acting as the restriction on morphisms constitutes a covariant
    functor for all $R \ge 0$.
\end{corollary}

Taking a closer look at the first part of the proof of
Proposition~\ref{prop:EntireOfMalleable} yields the following.
\begin{proposition}
    Let $\varrho \colon \liealg{g} \longrightarrow \Linear(V)$ be a
    Lie algebra representation and $R \ge 0$. Then $\varrho$ restricts
    to a Lie algebra representation on $\Entire_R(\varrho)$, whose
    space of entire vectors of order~$R$ is $\Entire_R(\varrho)$.
\end{proposition}
\begin{proof}
    We have already established that $\varrho$ restricts to a Lie
    algebra representation
    \begin{equation*}
        \varrho
        \colon
        \liealg{g}
        \longrightarrow
        \Linear
        \bigl(
            \Entire_R(\varrho)
        \bigr).
    \end{equation*}
    Let now $\seminorm{q} \in \cs(V)$ and $v \in
    \Entire_R(\varrho)$. By assumption, we know
    $\seminorm{p}_{r,\seminorm{q}}^{(R)}(v) < \infty$ for all $r \ge
    0$. Viewing the seminorm again as power series, this means that
    \begin{equation*}
        f
        \colon
        \field{C} \longrightarrow \field{C}, \quad
        f(z)
        \coloneqq
        \sum_{k=0}^{\infty}
        k!^{R-1}
        \cdot
        \sup_{\xi_1, \ldots, \xi_k \in \UnitBall}
        \seminorm{q}
        \bigl(
            \xi_1 \cdots \xi_k
            \acts
            v
        \bigr)
        \cdot
        z^k
    \end{equation*}
    is an entire holomorphic function. Taylor expanding around
    $r \in \field{C}$ leads to the expression
    \begin{equation*}
        f(z)
        =
        \sum_{n=0}^{\infty}
        \frac{(z-r)^n}{n!}
        \cdot
        f^{(n)}(r)
        =
        \sum_{n=0}^{\infty}
        \frac{(z-r)^n}{n!}
        \sum_{k=n}^{\infty}
        \frac{r^{k-n} \cdot k!^R}{(k-n)!}
        \sup_{\xi_1,\ldots,\xi_k \in \UnitBall}
        \seminorm{q}
        \bigl(
            \xi_1 \cdots \xi_k
            \acts
            v
        \bigr).
    \end{equation*}
    This results in the estimate
    \begin{align*}
        \seminorm{p}_{r_0, \seminorm{p}^{(R)}_{r,\seminorm{q}}}^{(R)}
        (v)
        &=
        \sum_{n=0}^{\infty}
        n!^{R-1}
        \cdot
        r_0^n
        \cdot
        \sup_{\xi_1, \ldots, \xi_n \in \UnitBall}
        \seminorm{p}_{r,\seminorm{q}}
        \bigl(
            \xi_1 \cdots \xi_n
            \acts
            v
        \bigr) \\
        &=
        \sum_{n=0}^{\infty}
        n!^{R-1}
        \cdot
        r_0^n
        \cdot
        \sup_{\xi_1, \ldots, \xi_n \in \UnitBall}
        \sum_{k=0}^{\infty}
        k!^{R-1}
        \cdot
        r^k
        \cdot
        \sup_{\eta_1, \ldots, \eta_{k} \in \UnitBall}
        \seminorm{q}
        \bigl(
            \eta_1 \cdots \eta_k
            \xi_1 \cdots \xi_n
            \acts
            v
        \bigr) \\
        &\le
        \sum_{n=0}^{\infty}
        n!^{R-1}
        \cdot
        r_0^n
        \cdot
        \sum_{k=0}^{\infty}
        k!^{R-1}
        \cdot
        r^k
        \cdot
        \sup_{\xi_1, \ldots, \xi_{n+k} \in \UnitBall}
        \seminorm{q}
        \bigl(
            \xi_1 \cdots \xi_{n+k}
            \acts
            v
        \bigr) \\
        &=
        \sum_{n=0}^{\infty}
        n!^{R-1}
        \cdot
        r_0^n
        \cdot
        \sum_{k=n}^{\infty}
        (k-n)!^{R-1}
        \cdot
        r^{k-n}
        \cdot
        \sup_{\xi_1, \ldots, \xi_k \in \UnitBall}
        \seminorm{q}
        \bigl(
            \xi_1 \cdots \xi_{k}
            \acts
            v
        \bigr) \\
        &=
        \sum_{n=0}^{\infty}
        \frac{r_0^n}{n!}
        \cdot
        \sum_{k=n}^{\infty}
        \binom{k}{n}^{-R}
        \cdot
        \frac{r^{k-n} \cdot k!^R}{(k-n)!}
        \cdot
        \sup_{\xi_1, \ldots, \xi_k \in \UnitBall}
        \seminorm{q}
        \bigl(
        \xi_1 \cdots \xi_{k}
        \acts
        v
        \bigr) \\
        &\le
        f(r_0 + r) \\
        &<
        \infty
    \end{align*}
    for all $r,r_0 \ge 0$. Hence, $v$ is also entire of order $R$ with
    respect to the $\Entire_R(\varrho)$-topology.
\end{proof}

We are now in a position to establish universal deformation formulas
for triples of the form \eqref{eq:EntireOfMalleableTriple}, resulting in continuous
deformed structures. Given a formal power series
\begin{equation}
    F_\hbar
    =
    \sum_{n=0}^{\infty}
    \frac{\hbar^n}{n!}
    \cdot
    F_n
    \in
    \bigl(
    \Universal(\liealg{g}) \tensor \Universal(\liealg{g})
    \bigr)\formal{\hbar},
\end{equation}
we write
\begin{equation}
    \varrho(F_\hbar)
    \coloneqq
    \sum_{n=0}^{\infty}
    \hbar^n \cdot \varrho(F_n)
\end{equation}
for the action of $F_\hbar$ on $(V \tensor V)\formal{\hbar}$. Lifting
$\mu$ to formal power series in a $\field{C}\formal{\hbar}$-linear
fashion, we may then define
\begin{equation}
    \mu_{F_\hbar}
    \coloneqq
    \mu \circ \varrho(F_\hbar)
    \colon
    (V \tensor V)\formal{\hbar}
    \longrightarrow
    V\formal{\hbar}.
\end{equation}
Finally, our analytic considerations will allow us to plug in a
concrete value of $\hbar \in \field{C}$, at least on the subalgebra of
entire vectors.
\begin{theorem}
    \label{thm:UniversalDeformationEntire}%
    Let
    \begin{equation}
        \mu
        \colon
        V \tensor W \longrightarrow X
    \end{equation}
    be a continuous malleable $\liealg{g}$-triple such that $X$ is
    sequentially complete Hausdorff. Moreover, let $R \ge 0$ and
    \begin{equation}
        F_\hbar
        =
        \sum_{n=0}^{\infty}
        \frac{\hbar^n}{n!}
        \cdot
        F_n
        \in
        \bigl(
            \Universal(\liealg{g}) \tensor \Universal(\liealg{g})
        \bigr)\formal{\hbar}
    \end{equation}
    be a formal power series such that for every $r > 0$,
    $\seminorm{q}_V \in \cs(V)$, $\seminorm{q}_W \in \cs(W)$ and
    compact set $K \subseteq \field{C}$, there exist
    $\seminorm{q}_V' \in \cs(V)$, $\seminorm{q}_W' \in \cs(W)$ and
    $t,C > 0$ such that
    \begin{equation}
        \label{eq:FormalTwistEquicontinuityEntire}
        \Bigl(
            \seminorm{p}_{r,\seminorm{q}_V}^{(R)}
            \tensor
            \seminorm{p}_{r,\seminorm{q}_W}^{(R)}
        \Bigr)
        \biggl(
            \frac{\hbar^n}{n!}
            F_n \acts (v \tensor w)
        \biggr)
        \le
        C
        \cdot
        \seminorm{p}_{tr,\seminorm{q}_V'}^{(R)}(v)
        \cdot
        \seminorm{p}_{tr,\seminorm{q}_W'}^{(R)}(w)
    \end{equation}
    for all $v \in \Entire_R(\varrho_V)$,
    $w \in \Entire_R(\varrho_W)$, $\hbar \in K$ and
    $n \in \field{N}_0$. Then
    \begin{equation}
        \bigl(
            \Entire_R(\varrho_V), \Entire_R(\varrho_W), \Entire_R(\varrho_X)
        \bigr)
    \end{equation}
    is a continuous $\liealg{g}$-triple with respect to
    $\mu_{F_\hbar} \coloneqq \mu \circ \varrho(F_\hbar)$ for all
    $\hbar \in \field{C}$. Moreover, the mapping
    \begin{equation}
        \mu_{F_\hbar}
        \colon
        K^\interior
        \times
        \Entire_R(\varrho_V)
        \tensor
        \Entire_R(\varrho_W)
        \longrightarrow
        \Entire_R(\varrho_X)
    \end{equation}
    is Fréchet holomorphic.
\end{theorem}
\begin{proof}
    Let $\seminorm{q}_X \in \cs(X)$, $v \in \Entire_R(\varrho_V)$,
    $w \in \Entire_R(\varrho_W)$. Moreover, it suffices to consider
    \begin{equation*}
        K
        \coloneqq
        \Ball_s(0)^\cl
        \subseteq
        \field{C}
    \end{equation*}
    to be a closed disk with radius $s > 0$. Invoking
    Proposition~\ref{prop:EntireOfMalleable}, we find continuous
    seminorms $\seminorm{q}_V \in \cs(V)$ and $\seminorm{q}_W \in
    \cs(W)$ such that
    \begin{equation*}
        \seminorm{p}_{r, \seminorm{q}_X}^{(R)}
        \bigl(
            \mu(z)
        \bigr)
        \le
        \bigl(
        \seminorm{p}_{2^R r, \seminorm{q}_V}^{(R)}
        \tensor
        \seminorm{p}_{2^R r, \seminorm{q}_W}^{(R)}
        \bigr)
        (z)
    \end{equation*}
    for all $r \ge 0$ and
    $z \in \Entire_R(\varrho_V) \tensor \Entire_R(\varrho_W)$.
    Combining this estimate with our equicontinuity assumption
    \eqref{eq:FormalTwistEquicontinuityEntire}, applied to the compact
    set $2K$, leads to
    \begin{align*}
        \sum_{n=0}^{\infty}
        \seminorm{p}_{r,\seminorm{q}_X}^{(R)}
        \biggl(
            \frac{\hbar^n}{n!}
            \mu
            \bigl(
                F_n \acts (v \tensor w)
            \bigr)
        \biggr)
        &\le
        \sum_{n=0}^{\infty}
        \bigl(
            \seminorm{p}_{2^R r,\seminorm{q}_V}^{(R)}
            \tensor
            \seminorm{p}_{2^R r,\seminorm{q}_W}^{(R)}
        \bigr)
        \biggl(
            \frac{\hbar^n}{n!}
            F_n
            \acts
            (v \tensor w)
        \biggr) \\
        &\le
        \sum_{n=0}^{\infty}
        2^{-n}
        \cdot
        C
        \cdot
        \seminorm{p}_{2^R tr,\seminorm{q}_V'}^{(R)}(v)
        \cdot
        \seminorm{p}_{2^R tr,\seminorm{q}_W'}^{(R)}(w) \\
        &=
        2 \cdot C
        \cdot
        \seminorm{p}^{(R)}_{2^R tr,\seminorm{q}_V'}(v)
        \cdot
        \seminorm{p}^{(R)}_{2^R tr,\seminorm{q}_W'}(w)
    \end{align*}
    for all $r \ge 0$. As $\mu(F_n \acts (v \tensor w)) \in
    \Entire_R(\varrho_{X})$ for all $n \in \N_0$ by
    Proposition~\ref{prop:EntireOfMalleable}, this means that the series
    \begin{equation*}
        \mu_{F_\hbar}
        (v \tensor w)
        \coloneqq
        \sum_{n=0}^{\infty}
        \frac{\hbar^n}{n!}
        \mu
        \bigl(
        F_n \acts (v \tensor w)
        \bigr)
    \end{equation*}
    converges within the sequentially complete Hausdorff space
    $\Entire_R(\varrho_X)$, see again
    Corollary~\ref{cor:EntireCompleteness}. Our estimate for the
    absolute convergence now implies a continuity estimate for
    $\mu_{F_\hbar}$ on factorizing tensors. Invoking the infimum
    argument from Proposition~\ref{prop:InfimumArgument} then yields
    the continuity of
    \begin{equation*}
        \mu_{F_\hbar}
        \colon
        \Entire_R(\varrho_V)
        \tensor
        \Entire_R(\varrho_W)
        \longrightarrow
        \Entire_R(\varrho_X).
    \end{equation*}
    Finally, combining the power series expansion with the continuity
    establishes the Fréchet holomorphy of $\mu_\hbar$ at once.
\end{proof}

The condition \eqref{eq:FormalTwistEquicontinuityEntire} facilitating
the proof of Theorem~\ref{thm:UniversalDeformationEntire} is perhaps a
bit mysterious and ostensibly difficult to verify. Taking a closer
look, we may make two simplifications.
\begin{remark}[Equicontinuity Condition]
    \label{rem:SmallRadius}%
    \ \vspace{-0.25cm}
    \begin{remarklist}
    \item \label{item:SmallRadiusLowerBound}%
        Recall that the seminorms
        $\seminorm{p}_{r,\seminorm{q}}^{(R)}$ are monotonically
        increasing with respect to $r > 0$. Thus, it suffices to
        demand \eqref{eq:FormalTwistEquicontinuityEntire} for radii
        $r > 0$ with $r > r_{\min}$ for some fixed minimal radius~$r_{\min}$.
        While this changes the radii for the seminorms for $r < r_{\min}$ on the
        right-hand side of \eqref{eq:FormalTwistEquicontinuityEntire}, they remain
        uniformly bounded for all $n \in \N_0$, resulting in no complications for the
        proof.

    \item \label{item:SmallRadiusCompact}%
        Again by monotonicity, it suffices to consider
        $K = \Ball_m(0)^\cl \subseteq \field{C}$ given by compact disks
        centered at the origin. Moreover, by rescaling this disk, an
        estimate of the form
        \begin{equation}
            \Bigl(
            \seminorm{p}_{r,\seminorm{q}_V}^{(R)}
            \tensor
            \seminorm{p}_{r,\seminorm{q}_W}^{(R)}
            \Bigr)
            \biggl(
                \frac{\hbar^n}{n!}
                F_n \acts (v \tensor w)
            \biggr)
            \le
            C
            \cdot
            T^n
            \cdot
            \seminorm{p}_{tr,\seminorm{q}_V'}^{(R)}(v)
            \cdot
            \seminorm{p}_{tr,\seminorm{q}_W'}^{(R)}(w)
        \end{equation}
        for some $T > 0$ not depending on $n \in \N_0$ and with remaining
        dependencies as before is sufficient to infer
        \eqref{eq:FormalTwistEquicontinuityEntire}.
    \end{remarklist}
\end{remark}

In the sequel, we call a pair of a continuous malleable
$\liealg{g}$-triple and a Drinfeld twist $F_\hbar$ \emph{entirely
  malleable} if they fulfil
\eqref{eq:FormalTwistEquicontinuityEntire}. This yields a full
subcategory $\categoryname{emTriple}_\liealg{g}$ of
$\categoryname{cmTriple}_\liealg{g}$. In this terminology, we have
constructed the following deformation functor.
\begin{corollary}
    The assignment
    \begin{equation}
        \label{eq:DeformationFunctorEntire}
        \mathcal{D}_{\hbar}
        \colon
        \categoryname{emTriple}_\liealg{g}
        \longrightarrow
        \categoryname{cTriple}_\liealg{g}
    \end{equation}
    sending $\mu \colon V \tensor W \longrightarrow X$ to
    \begin{equation}
        \mu_{F_\hbar}
        =
        \mu
        \circ
        \varrho(F_\hbar)
        \colon
        \Entire_R(V) \tensor \Entire_R(W)
        \longrightarrow
        \Entire_R(X)
    \end{equation}
    and acting by restriction on morphisms constitutes a covariant
    functor for all $\hbar \in \field{C}$.
\end{corollary}

The principal shortcoming of
Theorem~\ref{thm:UniversalDeformationEntire} now lies in the size of
$\Entire_R(\varrho)$. In Section~\ref{sec:Examples} we shall see that
the non-abelian examples require $R \ge 1$.
\begin{lemma}
    \label{lem:EntireNoGo}%
    Let $\varrho \colon \liealg{g} \longrightarrow \Linear(V)$ be a
    Lie algebra representation on a locally convex Hausdorff space
    $V$. If $\xi \in \liealg{g}$ and $v \in V \setminus \{0\}$ are
    such that $\varrho(\xi)v = \lambda \cdot v$ with
    $\lambda \in \C \setminus \{0\}$, then
    \begin{equation}
        v \notin \Entire_1(\varrho).
    \end{equation}
\end{lemma}
\begin{proof}
    Using the Hausdorff property, we find $\seminorm{q} \in \cs(V)$
    such that $\seminorm{q}(v) > 0$. Our assumption moreover
    necessitates $\xi \neq 0$. We estimate
    \begin{align*}
        \seminorm{p}_{r,\seminorm{q}}^{(1)}(v)
        &=
        \sum_{n=0}^{\infty}
        r^n
        \cdot
        \sup_{\xi_1, \ldots, \xi_n \in \UnitBall}
        \seminorm{q}
        \bigl(
            \xi_1 \cdots \xi_n
            \acts
            v
        \bigr) \\
        &\ge
        \sum_{n=0}^{\infty}
        \frac{r^n}{\norm{\xi}^n}
        \cdot
        \seminorm{q}
        \bigl(
            \xi^n
            \acts
            v
        \bigr) \\
        &=
        \seminorm{q}(v)
        \cdot
        \sum_{n=0}^{\infty}
        \biggl(
            \frac{r \cdot \abs{\lambda}}{\norm{\xi}}
        \biggr)^n,
    \end{align*}
    which diverges for all $r \ge \norm{\xi}/\abs{\lambda}$. Hence,
    $v \notin \Entire_1(\varrho)$.
\end{proof}

%% file: TeX/ConvergenceOfUDFAnalytic.tex

Throughout Section~\ref{subsec:ConvergenceOfUDFEntire}, we have often
employed that infinite radius of convergence is invariant under
rescaling. The idea, which allows us to treat more general vectors
such as the ones discussed within Lemma~\ref{lem:EntireNoGo}, is now
that the same is true about \emph{positive} radius of
convergence. That is to say, instead of the intersection as for the
entire vectors, we are lead to consider the union
\begin{equation}
    \Analytic_R(\varrho)
    =
    \bigcup_{r > 0}
    \Analytic_{R,r}(\varrho),
\end{equation}
whose natural topology is the one of a locally convex inductive limit
based on the canonical inclusions
\begin{equation}
    \label{eq:AnalyticFinalTopology}
    \Analytic_{R,r_0}(\varrho)
    \hookrightarrow
    \bigcup_{r > 0}
    \Analytic_{R,r}(\varrho)
\end{equation}
for all $r_0 > 0$. This is a particularly simple case of an inductive
system, whose definition we briefly recall for the convenience of the
reader.

Let $J$ be a directed set, which we view as a category with objects
given by the elements of $J$ and
\begin{equation}
    \Hom(\alpha,\beta)
    \coloneqq
    \begin{cases}
        \{*\}
        \; &\textrm{if} \;
        \alpha \earlier \beta, \\
        \emptyset
        \; &\textrm{if} \;
        \alpha \not\earlier \beta.
    \end{cases}
\end{equation}
From this point of view, an inductive system indexed by $J$ is a
covariant functor
\begin{equation}
    J
    \longrightarrow
    \categoryname{lcs}
\end{equation}
with values in the category of locally convex spaces. Unwrapping this
data, we get a pair
\begin{equation}
    \bigl(
        (V_\alpha)_{\alpha \in J},
        (\phi_{\alpha \beta})_{\alpha,\beta \in J}
    \bigr)
\end{equation}
of locally convex spaces $V_\alpha$ for all $\alpha \in J$ with
continuous linear connecting maps $V_\alpha \rightarrow V_\beta$
whenever $\alpha \earlier \beta$.  We only need the particularly
simple setup, in which each connecting map is a set inclusion as in
\eqref{eq:AnalyticFinalTopology}. In the sequel, we call such an
inductive system \emph{fully reduced}. The corresponding inductive
limit, i.e. the colimit within the category $\categoryname{lcs}$, may
then be realized as
\begin{equation}
    V
    \coloneqq
    \varinjlim
    V_\alpha
    \cong
    \bigcup_{\alpha \in J}
    V_\alpha
\end{equation}
endowed with the final locally convex topology arising from the
canonical inclusions
\begin{equation}
    \label{eq:InductiveLimitInclusion}
    V_\alpha \hookrightarrow V
    \qquad
    \textrm{for all }
    \alpha \in J.
\end{equation}
In the sequel, we may thus use the characteristic property of the
final topology to check continuity of linear mappings and seminorms
defined on inductive limits. Note that we are not assuming the
strictness of the limit: the inclusions
\eqref{eq:InductiveLimitInclusion} need not be topological
embeddings. That being said, by their continuity, we always know that
the subspace topology of $V_\alpha$ is coarser than its original
topology. As a result, many of the abstract results on strict
inductive limits will be unavailable to us. A systematic treatment of
locally convex inductive limits is \cite[§19]{koethe:1969a}. With this
in mind, we may proceed by defining the notion of \emph{inductive}
$\liealg{g}$-triples.
\begin{definition}[Inductive $\liealg{g}$-triples]
    \label{def:gTripleInductive}%
    Let $\liealg{g}$ be a Lie algebra. A $\liealg{g}$-triple
    \begin{equation}
        \mu
        \colon
        V \tensor W
        \longrightarrow
        X
    \end{equation}
    is called inductive if:
    \begin{definitionlist}
    \item \label{item:gTripleInductiveLimit}%
        The spaces $V,W$ and $X$ arise as locally convex inductive
        limits of fully reduced inductive systems
        $(V_\alpha)_{\alpha \in J_V}$, $(W_\beta)_{\beta \in J_W}$ and
        $(X_\gamma)_{\gamma \in J_X}$ consisting of invariant
        subspaces for $\varrho_V$, $\varrho_W$ and $\varrho_X$,
        respectively.

    \item \label{item:gTripleInductiveContinuity}%
        For every $\alpha \in J_V$ and $\beta \in J_W$ there exists an
        index $\gamma \in J_X$ such that
        \begin{equation}
            \label{eq:MalleableSystemContinuity}
            \mu
            \colon
            V_\alpha \tensor W_\beta
            \longrightarrow
            X_\gamma
        \end{equation}
        is well-defined and continuous.
    \end{definitionlist}
    A morphism of inductive $\liealg{g}$-triples is a morphism
    \begin{equation}
        (T_V,T_W,T_X)
        \colon
        (V,W,X)
        \longrightarrow
        (\tilde{V},\tilde{W},\tilde{X})
    \end{equation}
    of the underlying $\liealg{g}$-triples and
    $\tilde{\liealg{g}}$-triples consisting of continuous linear
    mappings such that for all $\alpha \in J_Z$ there exists another
    index $\tilde{\alpha} \in J_{\tilde{Z}}$ with
    $T_Z(Z_\alpha) \subseteq \tilde{Z}_{\tilde{\alpha}}$, where
    $Z \in \{V,W,X\}$.
\end{definition}

We denote the subcategory of inductive $\liealg{g}$-triples by
$\categoryname{indTriple}$. Moreover assuming malleability yields then
the subcategory $\categoryname{indmTriple}$, which we are ultimately
interested in. Our main result constitutes a universal deformation
formula in the form of a functor
\begin{equation}
    \categoryname{indmTriple}
    \longrightarrow
    \categoryname{indTriple}.
\end{equation}
Before developing the abstract theory, we discuss some first examples of
inductive $\liealg{g}$-triples.
\begin{example}[Constant inductive $\liealg{g}$-triples]
    \label{ex:ContinuousToInductiveTriple}
    Let $\liealg{g}$ be a Lie algebra and
    \begin{equation}
        \mu
        \colon
        V \tensor W
        \longrightarrow
        X
    \end{equation}
    be a $\liealg{g}$-triple. Set
    $J_V \coloneqq J_W \coloneqq J_X \coloneqq \{*\}$. Then
    \begin{equation}
        V_*
        \coloneqq
        V, \,
        W_*
        \coloneqq
        W, \,
        X_*
        \coloneqq
        X
    \end{equation}
    realizes $V$, $W$ and $X$ as inductive limits of a fully reduced
    systems. Trivially, each of the spaces $V_*$, $W_*$ and $X_*$ is
    invariant under the corresponding representation. Unwrapping the
    continuity condition from Definition~\ref{def:gTripleInductive},
    \ref{item:gTripleInductiveContinuity}, we finally see that it is
    simply asking for the continuity of the $\liealg{g}$-triple, see
    again Definition~\ref{def:gTriple},
    \ref{item:gTripleContinuity}. This way, we obtain functors
    \begin{equation}
        \label{eq:ContinuousToInductiveTriple}
        \categoryname{cTriple}
        \longrightarrow
        \categoryname{indTriple}
        \qquad \textrm{and} \qquad
        \categoryname{cmTriple}
        \longrightarrow
        \categoryname{indmTriple}.
    \end{equation}
\end{example}

A more concrete example arises from spaces of test functions.
\begin{example}
    \label{example:CompactlySupportedStuff}%
    Denote the space of all continuous complex-valued mappings defined
    on the real line $\field{R}$ by $\Continuous(\field{R})$. We endow
    this space with its usual topology of locally uniform convergence,
    which turns it into a Fréchet algebra with respect to the pointwise
    product. Consider the space of compactly supported continuous
    functions
    \begin{equation}
        \Continuous_c(\field{R})
        \coloneqq
        \bigl\{
            f
            \in
            \Continuous(\field{R})
            \;\big|\;
            \exists_{r > 0}
            \forall_{\abs{x} \ge r}
            \colon
            f(x)
            =
            0
        \bigr\}
    \end{equation}
    on the real line endowed with the locally convex inductive limit
    topology arising from the subspaces
    \begin{equation}
        \Continuous_n(\field{R})
        \coloneqq
        \bigl\{
            f
            \in
            \Continuous(\field{R})
            \;\big|\;
            \forall_{\abs{x} \ge n}
            \colon
            f(x)
            =
            0
        \bigr\}
    \end{equation}
    for all $n \in \N$. Here, we endow each
    $\Continuous_n(\field{R})$ with the subspace topology inherited
    from the further inclusion
    $\Continuous_n(\field{R}) \subseteq \Continuous(\field{R})$. Then
    $\Continuous_c(\field{R})$ is a $\Continuous(\field{R})$-module,
    whose module multiplication
    \begin{equation}
        \mu
        \colon
        \Continuous(\field{R})
        \tensor
        \Continuous_c(\field{R})
        \longrightarrow
        \Continuous_c(\field{R})
    \end{equation}
    fails to be continuous, see e.g. \cite[Prop.~4.1]{gloeckner:2006a}.
    However, its restrictions
    \begin{equation}
        \mu_n
        \colon
        \Continuous(\field{R})
        \tensor
        \Continuous_n(\field{R})
        \longrightarrow
        \Continuous_n(\field{R})
    \end{equation}
    are continuous for all $n \in \N$ by continuity of the
    multiplication within $\Continuous(\field{R})$. Hence, the same is
    true for its composition
    \begin{equation}
        \iota_n
        \circ
        \mu_n
        \colon
        \Continuous(\field{R})
        \tensor
        \Continuous_n(\field{R})
        \longrightarrow
        \Continuous_c(\field{R})
    \end{equation}
    with the canonical embedding
    $\iota_n \colon \Continuous_n(\field{R}) \longrightarrow
    \Continuous_c(\field{R})$. In passing, we note that in view of
    Proposition~\ref{prop:InfimumArgument}, this shows that the
    characteristic property of the final locally convex topology fails
    for bilinear mappings in general.  That being said, the triple
    \begin{equation}
        \bigl(
        \Continuous(\field{R}),
        \Continuous_c(\field{R}),
        \Continuous_c(\field{R})
        \bigr)
    \end{equation}
    constitutes an inductive $\liealg{g}$-triple for the zero Lie
    algebra. Endowing all spaces with more interesting representations
    then yields examples of inductive $\liealg{g}$-triples that do not
    arise from the functor \eqref{eq:ContinuousToInductiveTriple}. If
    the representations are by differential operators, then the
    triples are moreover malleable.
\end{example}

\begin{remark}
    \label{rem:gTripleInductiveContinuity}%
    Let
    $(T_V,T_W,T_X) \colon (V,W,X) \longrightarrow
    (\tilde{V},\tilde{W},\tilde{X})$ be a morphism of malleable
    triples along a continuous surjective Lie algebra morphism
    $\phi \colon \liealg{g} \longrightarrow \tilde{\liealg{g}}$.
    \begin{remarklist}
    \item Combining the characteristic property of the final locally
        convex topology with the continuity of the inclusions
        $X_\gamma \hookrightarrow X$ proves the separate continuity of
        the bilinear mapping
        \begin{equation}
            \mu_0
            \colon
            V \times W
            \longrightarrow
            X
        \end{equation}
        corresponding to $\mu$. In other words, for every $v_0 \in V$
        and $w_0 \in W$, the mappings
        \begin{equation}
            W
            \ni
            w
            \mapsto
            \mu(v_0,w)
            \in
            X
            \qquad \textrm{and} \qquad
            V
            \ni
            v
            \mapsto
            \mu(v,w_0)
            \in
            X
        \end{equation}
        are continuous.

    \item By the invariance
        condition~\ref{item:gTripleInductiveLimit}, we get Lie algebra
        representations
        \begin{align}
            &\varrho_{V,\alpha}
            \colon
            \liealg{g}
            \longrightarrow
            \Linear(V_\alpha) \\
            &\varrho_{W,\beta}
            \colon
            \liealg{g}
            \longrightarrow
            \Linear(W_\beta) \\
            &\varrho_{X,\gamma}
            \colon
            \liealg{g}
            \longrightarrow
            \Linear(X_\gamma)
        \end{align}
        for all $\alpha \in J_V$, $\beta \in J_W$ and
        $\gamma \in J_X$. Note that each $\varrho_{V,\alpha}(\xi)$ is
        indeed continuous by the characteristic property of the final
        locally convex topology.

    \item By continuity of $T_Z$ and the characteristic property of the
        final locally convex topology, the restrictions
        \begin{equation}
            T_Z
            \at[\Big]{Z_\alpha}
            \colon
            Z_\alpha
            \longrightarrow
            \tilde{Z}_{\tilde{\alpha}}
        \end{equation}
        are well-defined and continuous for all
        indices~$\alpha \in J_Z$ with corresponding
        $\tilde{\alpha} \in J_{\tilde{Z}}$ and $Z \in \{V,W,X\}$.
    \end{remarklist}
\end{remark}

Next, we introduce the appropriate notion of analytic vectors for
inductive $\liealg{g}$-triples. The principal idea is to consider
analytic vectors in the sense of Definition~\ref{def:AnalyticVectors}
of the constituent spaces, and then to pass back to the corresponding
locally convex inductive limit.
\begin{definition}[Analytic Vectors II]
    \label{def:AnalyticOfInductive}%
    Let
    \begin{equation}
        \varrho
        \colon
        \liealg{g}
        \longrightarrow
        \Linear(V)
    \end{equation}
    be a Lie algebra representation on a fully reduced inductive limit
    $V = \varinjlim V_\alpha$ consisting of $\varrho$-invariant
    subspaces and $R \ge 0$. Then we call
    \begin{equation}
        \label{eq:AnalyticOfSystem}
        \Analytic_R(\varrho)
        \coloneqq
        \varinjlim_{r > 0,\alpha \in J}
        \Analytic_{R,r}(\varrho_\alpha)
        \cong
        \bigcup_{r > 0, \alpha \in J}
        \Analytic_{R,r}(\varrho_\alpha)
    \end{equation}
    the analytic vectors of $V$ of order $R$ and endow it with the
    locally convex inductive limit topology arising from the doubly
    indexed inductive system $\{\Analytic_{R,r}(\varrho_\alpha)\}_{r > 0,
    \alpha \in J}$.
\end{definition}

\begin{lemma}
    Let
    \begin{equation}
        \varrho
        \colon
        \liealg{g}
        \longrightarrow
        \Linear(V)
    \end{equation}
    be a Lie algebra representation on a fully reduced inductive limit
    \begin{equation}
        V
        =
        \varinjlim_{\alpha \in J}
        V_\alpha
    \end{equation}
    consisting of $\varrho$-invariant subspaces and $R \ge 0$. If
    $V_\alpha$ is Hausdorff for all $\alpha \in J$, then so is
    $\Analytic_R(\varrho)$.
\end{lemma}
\begin{proof}
    Let $v \in V$. As the system is fully reduced, there thus exists
    $\alpha \in J$ and $r > 0$ such that
    $v \in \Analytic_{R,r}(\varrho_{\alpha})$, where
    \begin{equation*}
        \varrho_\alpha
        \coloneqq
        \varrho
        \at[\Big]{V_\alpha}.
    \end{equation*}
    As $V_\alpha$ is Hausdorff by assumption, we find
    $\seminorm{q} \in \cs(V_\alpha)$ such that $\seminorm{q}(v) >
    0$. By Remark~\ref{rem:AnalyticVectors},~\ref{item:AnalyticFiner},
    the seminorm $\seminorm{q}$ restricts to a continuous seminorm on
    $\Analytic_{R,s}(\varrho_\alpha)$ for all~$s > 0$ and thus $\seminorm{q} \in
    \cs(\Analytic_R(\varrho))$. Of course, it still fulfils
    $\seminorm{q}(v) > 0$.
\end{proof}

Next, we establish the compatibility of
Definition~\ref{def:AnalyticOfInductive} with
Definition~\ref{def:gTripleInductive}.
\begin{proposition}
    \label{prop:gTripleInductiveAnalytic}%
    Let
    \begin{equation}
        \mu
        \colon
        V \tensor W \longrightarrow X
    \end{equation}
    be a malleable inductive $\liealg{g}$-triple and $R \ge 0$. Then
    so is the triple
    \begin{equation}
        \bigl(
            \Analytic_R(\varrho_V),
            \Analytic_R(\varrho_W),
            \Analytic_R(\varrho_X)
        \bigr)
    \end{equation}
    with respect to the restriction of $\mu$. More precisely, for all
    radii $r_1, r_2 > 0$ and indices~${\alpha \in J_V}$,
    $\beta \in J_W$ there exists some $\gamma \in J_X$ such that for
    every continuous seminorm $\seminorm{q}_\gamma \in \cs(X_\gamma)$
    there are $\seminorm{q} \in \cs(V_\alpha)$ and
    $\seminorm{q}_\beta \in \cs(W_\beta)$ such that
    \begin{equation}
        \label{eq:gTripleAnalyticMultiplication}
        \seminorm{p}_{r,\seminorm{q}_\gamma}^{(R)}
        \bigl(
            \mu(v \tensor w)
        \bigr)
        \le
        \seminorm{p}_{r,\seminorm{q}_\alpha}^{(R)}(v)
        \cdot
        \seminorm{p}_{r,\seminorm{q}_\beta}^{(R)}(w)
    \end{equation}
    for all $r < 2^{-R} \min\{r_1,r_2\}$,
    $v \in \Analytic_{R,r_1}(\varrho_{V,\alpha})$ and
    $w \in \Analytic_{R,r_2}(\varrho_{W,\beta})$.
\end{proposition}
\begin{proof}
    We go through all aspects of Definition~\ref{def:gTriple} and
    Definition~\ref{def:gTripleInductive}. The malleability condition
    from \eqref{eq:gTripleCompatibility} is fulfilled by design, see
    again \eqref{eq:AnalyticOfSystem}.  The invariance of
    $\Analytic_{R,r}(\argument)$ under the corresponding
    representations was established in
    Proposition~\ref{prop:AnalyticInvariance}. It remains to verify
    the condition from Definition~\ref{def:gTripleInductive},
    \ref{item:gTripleInductiveContinuity}. To this end, let
    $\alpha \in J_V$, $\beta \in J_W$ and~$r_1,r_2 > 0$ as well as
    \begin{equation*}
        v
        \in
        \Analytic_{R,r_1}
        \bigl(
            \varrho_{V,\alpha}
        \bigr)
        \qquad \textrm{and} \qquad
        w
        \in
        \Analytic_{R,r_2}
        \bigl(
            \varrho_{W,\beta}
        \bigr).
    \end{equation*}
    By Definition~\ref{def:gTripleInductive},
    \ref{item:gTripleInductiveContinuity}, applied to $(V,W,X)$, there
    exists an index $\gamma \in J_X$ such that
    \begin{equation*}
        \mu
        \colon
        V_{\alpha} \tensor W_{\beta}
        \longrightarrow
        X_\gamma
    \end{equation*}
    is well-defined and continuous. In particular,
    $\mu(v \tensor w) \in X_\gamma$, and given
    $\seminorm{q}_\gamma \in \cs(X_\gamma)$, we find seminorms
    $\seminorm{q}_\alpha \in \cs(V_\alpha)$ and
    $\seminorm{q}_\beta \in \cs(W_\beta)$ such that
    \begin{equation*}
        \seminorm{q}_\gamma
        \bigl(
            \mu(v' \tensor w')
        \bigr)
        \le
        \seminorm{q}_\alpha(v')
        \cdot
        \seminorm{q}_\beta(w')
        \qquad
        \textrm{for all }
        v' \in V_\alpha,
        w' \in W_\beta.
    \end{equation*}
    Hence, a combination of the $\varrho$-invariance of $V_\alpha$ and
    $W_\beta$ with \eqref{eq:MalleableSystemLeibniz} yields
    \begin{align*}
        &\sup_{\xi_1, \ldots, \xi_n \in \UnitBall}
        \seminorm{q}_\gamma
        \bigl(
            \xi_1 \cdots \xi_n
            \acts
            \mu(v \tensor w)
        \bigr) \\
        &\le
        \sup_{\xi_1, \ldots, \xi_n \in \UnitBall}
        \sum_{k=0}^n
        \sum_{\sigma \in \Shuffle(k,n-k)}
        \seminorm{q}_\gamma
        \Bigl(
        \mu
        \bigl(
            (
            \underbrace
            {
                \xi_{\sigma(1)} \cdots \xi_{\sigma(k)}
                \acts
                v
            }_{\in V_\alpha}
            )
            \tensor
            (
            \underbrace
            {
                    \xi_{\sigma(k+1)} \cdots \xi_{\sigma(n)}
                    \acts
                    w
            }_{\in W_\beta}
            )
        \bigr)
        \Bigr) \\
        &\le
        \sup_{\xi_1, \ldots, \xi_n \in \UnitBall}
        \sum_{k=0}^n
        \sum_{\sigma \in \Shuffle(k,n-k)}
        \seminorm{q}_\alpha
        \bigl(
            \xi_{\sigma(1)} \cdots \xi_{\sigma(k)}
            \acts
            v
        \bigr)
        \cdot
        \seminorm{q}_\beta
        \bigl(
            \xi_{\sigma(k+1)} \cdots \xi_{\sigma(n)}
            \acts
            w
        \bigr) \\
        &\le
        \sum_{k=0}^n
        \binom{n}{k}
        \sup_{\xi_1, \ldots, \xi_k \in \UnitBall}
        \seminorm{q}_\alpha
        \bigl(
            \xi_1 \cdots \xi_k
            \acts
            v
        \bigr)
        \cdot
        \sup_{\xi_1, \ldots, \xi_{n-k} \in \UnitBall}
        \seminorm{q}_\beta
        \bigl(
            \xi_1 \cdots \xi_{n-k}
            \acts
            w
        \bigr).
    \end{align*}
    Consequently, if $r < r_0 \coloneqq 2^{-R} \cdot \min\{r_1, r_2\}$, then
    \begin{align*}
        &\seminorm{p}_{r,\seminorm{q}_\gamma}^{(R)}
        \bigl(
            \mu(v \tensor w)
        \bigr) \\
        &=
        \sum_{n=0}^{\infty}
        n!^{R-1}
        \cdot
        r^n
        \sup_{\xi_1, \ldots, \xi_n \in \UnitBall}
        \seminorm{q}_\gamma
        \bigl(
            \xi_1 \cdots \xi_n \acts
            \mu(v \tensor w)
        \bigr) \\
        &\le
        \sum_{n=0}^{\infty}
        n!^{R-1}
        \cdot
        r^n
        \sum_{k=0}^{n}
        \binom{n}{k}
        \sup_{\xi_1, \ldots, \xi_k \in \UnitBall}
        \seminorm{q}_\alpha
        (\xi_1 \cdots \xi_k \acts v)
        \sup_{\xi_1, \ldots, \xi_{n-k} \in \UnitBall}
        \seminorm{q}_\beta
        (\xi_1 \cdots \xi_{n-k} \acts w) \\
        &=
        \sum_{n=0}^{\infty}
        \sum_{k=0}^{n}
        \binom{n}{k}^{R}
        \Bigl(
            k!^{R-1}
            \cdot
            r^k
            \sup_{\xi_1, \ldots, \xi_k \in \UnitBall}
            \seminorm{q}_\alpha
            (\xi_1 \cdots \xi_k \acts v)
        \Bigr) \\
        &\hspace{5cm}
        \Bigl(
            (n-k)!^{R-1}
            \cdot
            r^{n-k}
            \sup_{\xi_1, \ldots, \xi_{n-k}\in \UnitBall}
            \seminorm{q}_\beta
            (\xi_1 \cdots \xi_{n-k} \acts w)
        \Bigr) \\
        &\le
        \sum_{n=0}^{\infty}
        \sum_{k=0}^{n}
        2^{nR}
        \Bigl(
            k!^{R-1}
            \cdot
            r^k
            \sup_{\xi_1, \ldots, \xi_k \in \UnitBall}
            \seminorm{q}_\alpha
            (\xi_1 \cdots \xi_k \acts v)
        \Bigr) \\
        &\hspace{5cm}
        \Bigl(
            (n-k)!^{R-1}
            \cdot
            r^{n-k}
            \sup_{\xi_1, \ldots, \xi_{n-k} \in \UnitBall}
            \seminorm{q}_\beta
            (\xi_1 \cdots \xi_{n-k} \acts w)
        \Bigr) \\
        &=
        \biggl(
            \sum_{n=0}^{\infty}
            n!^{R-1}
            \cdot
            (2^{R} r)^n
            \sup_{\xi_1, \ldots, \xi_n \in \UnitBall}
            \seminorm{q}_\alpha
            (\xi_1 \cdots \xi_k \acts v)
        \biggr) \\
        &\hspace{5cm}
        \biggl(
            \sum_{n=0}^{\infty}
            n!^{R-1}
            \cdot
            (2^{R} r)^n
            \sup_{\xi_1, \ldots, \xi_n \in \UnitBall}
            \seminorm{q}_\beta
            (\xi_1 \cdots \xi_n \acts w)
        \biggr) \\
        &=
        \seminorm{p}_{2^R r, \seminorm{q}_\alpha}^{(R)}(v)
        \cdot
        \seminorm{p}_{2^R r, \seminorm{q}_\beta}^{(R)}(w),
    \end{align*}
    where we have used the rather crude estimate
    $\binom{n}{k} \le 2^n$ and the Cauchy product formula for infinite
    series. That is to say,
    $\mu(v \tensor w) \in
    \Analytic_{R,r_0}(\varrho_{X,\gamma})$. Varying the vectors~$v,w$ for fixed
    radii $r_1,r_2 > 0$ and invoking Proposition~\ref{prop:InfimumArgument}, we
    have moreover established the continuity of
    \begin{equation*}
        \mu
        \colon
        \Analytic_{R,r_1}(\varrho_{V,\alpha})
        \tensor
        \Analytic_{R,r_2}(\varrho_{W,\beta})
        \longrightarrow
        \Analytic_{R,r_0}(\varrho_{X,\gamma}).
    \end{equation*}
    This completes the proof.
\end{proof}

Once again, we may enhance the statement to a functor.
\begin{corollary}
    The assignment
    \begin{equation}
        \Analytic_R
        \colon
        \categoryname{indmTriple}
        \longrightarrow
        \categoryname{indmTriple},
        \quad
        \Analytic_R
        \bigl(
            V,W,X
        \bigr)
        \coloneqq
        \bigl(
            \Analytic_R(\varrho_V),
            \Analytic_R(\varrho_W),
            \Analytic_R(\varrho_X)
        \bigr)
    \end{equation}
    acting as the restriction on morphisms constitutes a covariant
    functor for all $R \ge 0$.
\end{corollary}
\begin{proof}
    By Proposition~\ref{prop:gTripleInductiveAnalytic}, we know that
    restricting $\mu$ results in another malleable inductive
    $\liealg{g}$-triple. It remains to check the compatibility of
    $\Analytic_R$ with morphisms. Let
    \begin{equation*}
        (T_V,T_W,T_X)
        \colon
        (V,W,X)
        \longrightarrow
        (\tilde{V},\tilde{W},\tilde{X})
    \end{equation*}
    be a morphism of malleable inductive $\liealg{g}$-triples along a
    surjective Lie algebra morphism
    \begin{equation*}
        \phi
        \colon
        \liealg{g}
        \longrightarrow
        \tilde{\liealg{g}}.
    \end{equation*}
    As $\Analytic_R(V) \subseteq V$, we may consider the restriction
    \begin{equation*}
        T_V
        \at[\Big]{\Analytic_R(V)}
        \colon
        \Analytic_R(V)
        \longrightarrow
        \tilde{V}.
    \end{equation*}
    By a combination of Corollary~\ref{cor:EntireFunctoriality} with
    Definition~\ref{def:AnalyticOfInductive}, this restriction
    co-restricts to the space~$\Analytic_R(\tilde{V})$. Repeating the
    arguments for $W$ and $X$ and suppressing the restrictions, we
    thus get a triple of continuous linear mappings
    \begin{equation*}
        (T_V,T_W,T_X)
        \colon
        \bigl(
            \Analytic_R(\varrho_V),
            \Analytic_R(\varrho_W),
            \Analytic_R(\varrho_X)
        \bigr)
        \longrightarrow
        \bigl(
            \Analytic_R(\tilde{\varrho}_{\tilde{V}}),
            \Analytic_R(\tilde{\varrho}_{\tilde{W}}),
            \Analytic_R(\tilde{\varrho}_{\tilde{X}})
        \bigr),
    \end{equation*}
    where the continuity follows the universal property of the final
    locally convex topology combined with the continuity of the
    canonical inclusions and the original morphisms $T_V,T_W,T_X$, see
    again Remark~\ref{rem:gTripleInductiveContinuity}. Let now
    $Z \in \{V,W,X\}$, $\alpha \in J_Z$ and $r > 0$. As the
    triple~$(T_V,T_W,T_X)$ is a morphism of inductive
    $\liealg{g}$-triples, there exists a corresponding
    index~$\tilde{\alpha} \in J_{\tilde{Z}}$ with
    $T_Z(Z_\alpha) \subseteq \tilde{Z}_{\tilde{\alpha}}$. Invoking
    again Lemma~\ref{lem:AnalyticFunctoriality}, we moreover find
    $s > 0$ with
    \begin{equation*}
        T_Z
        \bigl(
            \Analytic_{R,r}(\varrho_{Z_\alpha})
        \bigr)
        \subseteq
        \Analytic_{R,s}(\tilde{\varrho}_{\tilde{Z}_{\tilde{\alpha}}}).
    \end{equation*}
    This is the compatibility condition with the grading for the
    restricted triple. Finally, as restriction respects algebraic
    relations, it is clear that \eqref{eq:gTriplesMorphism} and
    \eqref{eq:gTriplesMorphismMalleability} hold for the restricted
    morphisms.
\end{proof}

Using the language we have established, the abstract deformation
result for analytic vectors of malleable inductive
$\liealg{g}$-triples now takes the following form:
\begin{theorem}
    \label{thm:UniversalDeformationAnalytic}%
    Let
    \begin{equation}
        \mu
        \colon
        V \tensor W \longrightarrow X
    \end{equation}
    be an inductive malleable $\liealg{g}$-triple such that $X_\gamma$
    is sequentially complete Hausdorff for all indices
    $\gamma \in J_X$.  Moreover, let $R \ge 0$ and
    \begin{equation}
        F_\hbar
        =
        \sum_{n=0}^{\infty}
        \frac{\hbar^n}{n!}
        \cdot
        F_n
        \in
        \bigl(
        \Universal(\liealg{g}) \tensor \Universal (\liealg{g})
        \bigr)\formal{\hbar}
    \end{equation}
    be a formal power series such that for every $\alpha \in J_V$ and
    $\beta \in J_W$, $r_1,r_2 > 0$,
    $\seminorm{q}_\alpha \in \cs(V_\alpha)$,
    $\seminorm{q}_\beta \in \cs(W_\beta)$ and compact set
    $K \subseteq \field{C}$, there exist
    $\seminorm{q}_\alpha' \in \cs(V_\alpha)$,
    $\seminorm{q}_\beta' \in \cs(W_\beta)$ and $t,C > 0$ such that
    \begin{equation}
        \label{eq:FormalTwistEquicontinuityAnalytic}
        \Bigl(
            \seminorm{p}_{r,\seminorm{q}_\alpha}^{(R)}
            \tensor
            \seminorm{p}_{r,\seminorm{q}_\beta}^{(R)}
        \Bigr)
        \biggl(
            \frac{\hbar^n}{n!}
            F_n \acts (v \tensor w)
        \biggr)
        \le
        C
        \cdot
        \seminorm{p}_{tr,\seminorm{q}_\alpha'}^{(R)}(v)
        \cdot
        \seminorm{p}_{tr,\seminorm{q}_\beta'}^{(R)}(w)
    \end{equation}
    for all $v \in \Analytic_{R,r_1}(\varrho_{V,\alpha})$,
    $w \in \Analytic_{R,r_2}(\varrho_{W,\beta})$,
    $r < r_0 \coloneqq 2^{-R}/t \cdot \min\{r_1,r_2\}$, $\hbar \in K$
    and $n \in \field{N}_0$. Then
    \begin{equation}
        \bigl(
            \Analytic_R(V), \Analytic_R(W), \Analytic_R(X)
        \bigr)
    \end{equation}
    is an inductive $\liealg{g}$-triple with respect to
    $\mu_{F_\hbar} \coloneqq \mu \circ \varrho(F_\hbar)$ for all
    $\hbar \in \field{C}$. Moreover, the mapping
    \begin{equation}
        \mu_\hbar
        \colon
        K^\interior
        \times
        \Analytic_{R,r_1}(\varrho_{V,\alpha})
        \tensor
        \Analytic_{R,r_2}(\varrho_{W,\beta})
        \longrightarrow
        \Analytic_{R,r_0}(\varrho_{X,\gamma})
    \end{equation}
    is Fréchet holomorphic.
\end{theorem}
\begin{proof}
    Let $\alpha \in J_V$, $\beta \in J_W$, $r_1,r_2 > 0$ and
    $K \coloneqq \Ball_s(0)^\cl \subseteq \field{C}$ be a closed disk
    with radius $s > 0$. Invoking
    Proposition~\ref{prop:gTripleInductiveAnalytic}, we find an index
    $\gamma \in J_X$ such that for every
    $\seminorm{q}_\gamma \in \cs(X_\gamma)$ there are
    $\seminorm{q}_\alpha \in \cs(V_\alpha)$ and
    $\seminorm{q}_\beta \in \cs(W_\beta)$ such that
    \begin{equation*}
        \seminorm{p}_{r,\seminorm{q}_\gamma}^{(R)}
        \bigl(
            \mu(v \tensor w)
        \bigr)
        \le
        \seminorm{p}_{2^R r, \seminorm{q}_\alpha}^{(R)}
        (v)
        \cdot
        \seminorm{p}_{2^R r, \seminorm{q}_\beta}^{(R)}
        (w)
    \end{equation*}
    for all $r < r_0 \coloneqq 2^{-R} \min\{r_1,r_2\}$,
    $v \in \Analytic_{R,r_1}(\varrho_{V,\alpha})$ and
    $w \in \Analytic_{R,r_2}(\varrho_{W,\beta})$.  By
    Proposition~\ref{prop:InfimumArgument}, this implies
    \begin{equation*}
        \seminorm{p}_{r,\seminorm{q}_\gamma}^{(R)}
        \bigl(
            \mu(x)
        \bigr)
        \le
        \bigl(
            \seminorm{p}_{2^R r, \seminorm{q}_\alpha}^{(R)}
            \tensor
            \seminorm{p}_{2^R r, \seminorm{q}_\alpha}^{(R)}
        \bigr)(x)
        \qquad
        \textrm{for all }
        x
        \in
        \Analytic_{R,r_1}(\varrho_{V,\alpha})
        \tensor
        \Analytic_{R,r_2}(\varrho_{W,\beta})
    \end{equation*}
    and $r < r_0$. Combining this estimate with our equicontinuity
    assumption \eqref{eq:FormalTwistEquicontinuityAnalytic}, applied
    to the compact set $2K$, leads to
    \begin{align*}
        \sum_{n=0}^{\infty}
        \seminorm{p}_{r,\seminorm{q}_\gamma}^{(R)}
        \biggl(
        \frac{\hbar^n}{n!}
        \mu
        \bigl(
            F_n \acts (v \tensor w)
        \bigr)
        \biggr)
        &\le
        \sum_{n=0}^{\infty}
        \bigl(
            \seminorm{p}_{2^R r,\seminorm{q}_\alpha}^{(R)}
            \tensor
            \seminorm{p}_{2^R r,\seminorm{q}_\beta}^{(R)}
        \bigr)
        \biggl(
            \frac{\hbar^n}{n!}
            F_n
            \acts
            (v \tensor w)
        \biggr) \\
        &\le
        \sum_{n=0}^{\infty}
        2^{-n}
        \cdot
        C
        \cdot
        \seminorm{p}_{2^R tr,\seminorm{q}_\alpha'}^{(R)}(v)
        \cdot
        \seminorm{p}_{2^R tr,\seminorm{q}_\beta'}^{(R)}(w) \\
        &\le
        2 \cdot C
        \cdot
        \seminorm{p}^{(R)}_{r_1,\seminorm{q}_\alpha'}(v)
        \cdot
        \seminorm{p}^{(R)}_{r_2,\seminorm{q}_\beta'}(w),
    \end{align*}
    for all $r < r_0/t$. As
    $\mu(F_n \acts (v \tensor w)) \in
    \Analytic_{R,r_0}(\varrho_{X,\gamma})$ for all $n \in \N_0$, this
    means that the series
    \begin{equation*}
        \mu_{F_\hbar}
        (v \tensor w)
        \coloneqq
        \sum_{n=0}^{\infty}
        \frac{\hbar^n}{n!}
        \mu
        \bigl(
        F_n \acts (a \tensor b)
        \bigr)
    \end{equation*}
    converges within the sequentially complete Hausdorff space
    $\Analytic_{R,r_0/t}(\varrho_{X,\gamma})$, see again
    Proposition~\ref{prop:AnalyticCompleteness}. Our estimate for the
    absolute convergence now implies a continuity estimate for the
    linear mapping~$\mu_{F_\hbar}$ on factorizing tensors. Invoking
    once again Proposition~\ref{prop:InfimumArgument} then yields the
    continuity of
    \begin{equation*}
        \mu_{F_\hbar}
        \colon
        \Analytic_{R,r_1}(\varrho_{V,\alpha})
        \tensor
        \Analytic_{R,r_2}(\varrho_{W,\beta})
        \longrightarrow
        \Analytic_{R,r_0/t}(\varrho_{X,\gamma}).
    \end{equation*}
    This is precisely the condition for the multiplication in
    Definition~\ref{def:gTripleInductive},
    \ref{item:gTripleInductiveContinuity}.  Finally, combining the
    power series expansion with the continuity, we get the Fréchet
    holomorphicity of $\mu_\hbar$ as before.
\end{proof}

Our discussion from Remark~\ref{rem:SmallRadius} applies to the
equicontinuity condition \eqref{eq:FormalTwistEquicontinuityAnalytic},
as well. That is to say, it suffices to consider radii $r > 0$ with
\begin{equation}
    \label{eq:SmallRadiusAnalytic}
    0
    <
    r_{\min}
    \le
    r
    <
    r_0
    =
    \frac{\min\{r_1,r_2\}}{2^R \cdot t}
\end{equation}
and some fixed minimal radius~$r_{\min}$. Moreover, in view of
Example~\ref{ex:ContinuousToInductiveTriple}, the
condition~\eqref{eq:FormalTwistEquicontinuityAnalytic} is stronger
than \eqref{eq:FormalTwistEquicontinuityEntire}. Again, we may
interpret the deformation as a functor by including the Drinfeld
twist of $\liealg{g}$ into the data to define the subcategory
$\categoryname{indemTriple}_\liealg{g}$ of \emph{entirely malleable}
inductive $\liealg{g}$-triples for a fixed Lie algebra $\liealg{g}$.
\begin{corollary}
    The assignment
    \begin{equation}
        \label{eq:DeformationFunctorInductive}
        \mathcal{D}_{\hbar}
        \colon
        \categoryname{indemTriple}_\liealg{g}
        \longrightarrow
        \categoryname{indTriple}_\liealg{g}
    \end{equation}
    sending $\mu \colon V \tensor W \longrightarrow X$ to
    \begin{equation}
        \mu_{F_\hbar}
        =
        \mu
        \circ
        \varrho(F_\hbar)
        \colon
        \Analytic_R(V) \tensor \Analytic_R(W)
        \longrightarrow
        \Analytic_R(X)
    \end{equation}
    and acting by restriction on morphisms constitutes a covariant
    functor for all $\hbar \in \field{C}$.
\end{corollary}
In other words, the construction of $\mu_{F_\hbar}$ deserves the
name \emph{universal} deformation formula also in this analytic
context.

Important cases of inductive malleable $\liealg{g}$-triples are now
the following.
\begin{example}
    \label{ex:InductiveAlgebra}%
    Let $(\algebra{A}_\bullet, \mu)$ be a filtered associative algebra
    indexed by a directed set $J$. Assume moreover that each
    $\algebra{A}_\alpha$ carries a locally convex topology such that:
    \begin{examplelist}
    \item Whenever $\alpha \earlier \beta$, then the inclusion
        $\algebra{A}_\alpha \subseteq \algebra{A}_\beta$ induced by
        the filtration is continuous.

    \item For every $\alpha,\beta \in J$ there exists a $\gamma \in J$
        such that
        \begin{equation}
            \mu
            \colon
            \algebra{A}_\alpha \tensor \algebra{A}_\beta
            \longrightarrow
            \algebra{A}_\gamma
        \end{equation}
        is continuous and well-defined.
    \end{examplelist}
    Using both assumptions allows us to realize
    \begin{equation}
        \algebra{A}_\bullet
        =
        \bigcup_{\alpha \in J}
        \algebra{A}_\alpha
        =
        \varinjlim
        \algebra{A}_\alpha
    \end{equation}
    as a fully reduced locally convex inductive limit.  The important
    point to note is that we do \emph{not} assume continuity of $\mu$
    as a bilinear map on the inductive limit $\algebra{A}_\bullet$
    itself. As we have seen, this would exclude too many interesting
    examples. Finally, let
    \begin{equation}
        \varrho
        \colon
        \liealg{g}
        \longrightarrow
        \Der
        \bigl(
            \algebra{A}_\bullet
        \bigr)
    \end{equation}
    be a representation of a Lie algebra by derivations of $\mu$ such
    that each $\algebra{A}_\alpha$ is $\varrho$-invariant. Then the
    triple
    $(\algebra{A}_\bullet, \algebra{A}_\bullet, \algebra{A}_\bullet)$
    constitutes an inductive malleable $\liealg{g}$-triple with
    respect to~$\mu$.

    Indeed, most properties are built into our
    assumptions and the derivation condition corresponds precisely to
    \eqref{eq:gTripleCompatibility}. The corresponding deformation
    $\mu_\hbar$ of $\mu$ then constitutes another associative product
    on $\algebra{A}$. To see this, fix
    $a \in \Analytic_{R,r_a}(\varrho_\alpha)$,
    $b \in \Analytic_{R,r_b}(\varrho_\beta)$ and
    $c \in \Analytic_{R,r_c}(\varrho_\gamma)$ as well as a compact
    disk $K \coloneqq \Ball_s(0)^\cl \subseteq \field{C}$.  Then,
    invoking Theorem~\ref{thm:UniversalDeformationAnalytic} and using
    that our index set~$J$ is a direction to pass to a joint upper
    bound, we find a radius $r > 0$ and an index~$\delta \in J$ such
    that
    \begin{equation}
        a
        \star_{F_\hbar}
        (b \star_{F_{\hbar'}} c)
        \in
        \Analytic_{R,r}(\varrho_\delta)
        \ni
        (a \star_{F_\hbar} b)
        \star_{F_{\hbar'}}
        c
    \end{equation}
    for all $\hbar,\hbar' \in K$. We view both sides as separately
    holomorphic functions
    \begin{equation}
        f_L, f_R
        \colon
        K^\interior \times K^\interior
        \longrightarrow
        \Analytic_{R,r}(\varrho_\delta).
    \end{equation}
    By the vector-valued Hartog's Theorem this implies the Gâteaux
    holomorphicity of $f_L$ and $f_R$. In particular, evaluating on the
    diagonal yields Gâteaux holomorphic mappings
    \begin{equation}
        K^\interior
        \ni
        \hbar
        \mapsto
        f_L(\hbar,\hbar)
        \qquad \textrm{and} \qquad
        K^\interior
        \ni
        \hbar
        \mapsto
        f_R(\hbar,\hbar).
    \end{equation}
    As such, they possess unique local Taylor expansions, see e.g.
    \cite[Sec.~3.1]{dineen:1999a}. Due to the finite-dimensionality
    of $K^\interior$, these take the form of vector-valued power
    series. Recall now that, as formal power series, we know
    \begin{equation}
        f_L(\hbar,\hbar)
        =
        f_R(\hbar,\hbar),
    \end{equation}
    which thus also holds for the corresponding holomorphic
    functions. That is to say,
    \begin{equation}
        a
        \star_{F_\hbar}
        (b \star_{F_{\hbar}} c)
        =
        (a \star_{F_\hbar} b)
        \star_{F_{\hbar}}
        c
    \end{equation}
    and varying $a,b,c$ establishes the associativity of
    $\star_\hbar$.
\end{example}
\begin{example}
    \label{ex:InductiveModule}%
    Generalizing Example~\ref{ex:InductiveAlgebra}, we may consider
    again a filtered associative algebra
    $(\algebra{A}_\alpha)_{\alpha \in J_\algebra{A}}$. Instead of
    focusing on the multiplication of $\algebra{A}$, we consider a
    filtered left-module
    \begin{equation}
        \mu
        \colon
        \algebra{A}_\bullet
        \tensor
        (\module{M}_\beta)_{\beta \in J_\module{M}}
        \longrightarrow
        (\module{M}_\beta)_{\beta \in J_\module{M}}
    \end{equation}
    over $\algebra{A}$. Similar to before, we assume that the
    constituent spaces carry locally convex topologies such that:
    \begin{examplelist}
    \item Whenever $\alpha \earlier \beta$ for
        $\alpha,\beta \in J_\algebra{A}$, then the inclusion
        $\algebra{A}_\alpha \subseteq \algebra{A}_\beta$ induced by
        the filtration is continuous.

    \item Whenever $\alpha \earlier \beta$ for
        $\alpha,\beta \in J_\module{M}$, then the inclusion
        $\module{M}_\alpha \subseteq \module{M}_\beta$ induced by the
        filtration is continuous.

    \item For every $\alpha \in J_\algebra{A}$ and
        $\beta \in J_\module{M}$ there exists a
        $\gamma \in J_\algebra{A}$ such that
        \begin{equation}
            \mu
            \colon
            \algebra{A}_\alpha \tensor \module{M}_\beta
            \longrightarrow
            \module{M}_\gamma
        \end{equation}
        is continuous and well-defined. We will abbreviate
        $a \cdot m = \mu(a, m)$ in the following.
    \end{examplelist}
    As before, this allows us to realize
    \begin{equation}
        \algebra{A}_\bullet
        =
        \varinjlim_{\alpha \in J_\algebra{A}}
        \algebra{A}_\alpha
        \qquad
        \textrm{and}
        \qquad
        \module{M}_\bullet
        =
        \varinjlim_{\beta \in J_\module{M}}
        \module{M}
    \end{equation}
    as fully reduced locally convex inductive limits. Finally, two Lie
    algebra representations
    \begin{equation}
        \varrho_\algebra{A}
        \colon
        \liealg{g}
        \longrightarrow
        \Linear
        \bigl(
            \algebra{A}_\bullet
        \bigr)
        \qquad
        \textrm{and}
        \qquad
        \varrho_\module{M}
        \colon
        \liealg{g}
        \longrightarrow
        \Linear
        \bigl(
            \module{M}_\bullet
        \bigr)
    \end{equation}
    are compatible in the sense of \eqref{eq:gTripleCompatibility} iff
    \begin{equation}
        \label{eq:ModuleMalleability}
        \varrho_{\module{M}}(\xi)
        (a \cdot m)
        =
        a \cdot
        \bigl(
        \varrho_{\module{M}}(\xi)m
        \bigr)
        +
        \bigl(
        \varrho_{\algebra{A}}(\xi)a
        \bigr)
        \cdot
        m
    \end{equation}
    for all $a \in \algebra{A}_\bullet$, $m \in \module{M}_\bullet$
    and $\xi \in \liealg{g}$. Of course, we also have to assume that
    the representations respect the filtrations, but then we obtain
    another inductive malleable $\liealg{g}$-triple
    $(\algebra{A}_\bullet, \module{M}_\bullet, \module{M}_\bullet)$.
    The corresponding deformation $\mu_\hbar$ then corresponds to a
    deformation of the module structure by analogous arguments to the
    ones used in Example~\ref{ex:InductiveAlgebra}. Note that similar
    assumptions allow us to deform right-modules as well as bi-modules
    over two different algebras as well.
\end{example}

%% file: TeX/ExamplesAbelian.tex

We begin with deformations induced by commuting derivations, which we
model as an action
\begin{equation}
    \varrho
    \colon
    \liealg{g}
    \longrightarrow
    \Der(\algebra{A})
\end{equation}
of an abelian Lie algebra $\liealg{g}$ on a locally convex algebra
$\algebra{A}$.  In this case, the universal enveloping algebra
$\Universal(\liealg{g})$ is just the symmetric tensor algebra
$\Sym(\liealg{g})$. We consider $\Sym(\liealg{g})$ as a commutative
bialgebra with space of primitive elements given by
$\Sym^1 (\liealg{g}) = \liealg{g}$. Then
\cite[Theorem~2.1]{giaquinto.zhang:1998a} states that for every
$r \in \liealg{g} \tensor \liealg{g}$, the formal exponential
\begin{equation}
    \label{eq:ExponentialOfCommutingDerivations}
    F_\hbar
    \coloneqq
    \exp(\hbar r)
    \coloneqq
    \sum_{n=0}^{\infty}
    \frac{\hbar^n}{n!}
    \cdot
    r^n
    \in
    \bigl(\Sym (\liealg{g}) \tensor \Sym (\liealg{g})\bigr)
    \formal{\hbar},
\end{equation}
constitutes a Drinfeld twist. Here, the powers $r^n$ are taken
componentwise, and similarly~$r^n$ acts on
$\algebra{A} \tensor \algebra{A}$ by the action of each component. We
decompose
\begin{equation}
    r
    =
    \sum_{k=1}^{N}
    r_k \tensor s_k
\end{equation}
into factorizing tensors with $r_k,s_k \in \liealg{g}$ for
$k=1,\ldots,N$ and set
\begin{equation}
    c
    \coloneqq
    \max_{k=1,\ldots,N}
    \norm{r_k}
    +
    \max_{k=1,\ldots,N}
    \norm{s_k}.
\end{equation}
In particular, we get
\begin{equation}
    \label{eq:TwistAbelianPowers}
    r^n
    =
    \sum_{k_1,\ldots,k_n=1}^N
    (r_{k_1} \cdots r_{k_n})
    \tensor
    (s_{k_1} \cdots s_{k_n})
    \qquad
    \textrm{for all }
    n \in \N_0.
\end{equation}
Returning from derivations acting on algebras to $\liealg{g}$-triples, we get the
following.
\begin{proposition}
    Let $R \coloneqq 1/2$ and let
    \begin{equation}
        \mu
        \colon
        V \tensor W
        \longrightarrow
        X
    \end{equation}
    be a malleable inductive $\liealg{g}$-triple. Then the Drinfeld
    twist \eqref{eq:ExponentialOfCommutingDerivations} fulfils the
    equicontinuity conditions
    \eqref{eq:FormalTwistEquicontinuityEntire} and
    \eqref{eq:FormalTwistEquicontinuityAnalytic}.
\end{proposition}
\begin{proof}
    Let $\alpha \in J_V$, $\beta \in J_W$, $t_1,t_2 > 0$,
    $\seminorm{q}_\alpha \in \cs(V_\alpha)$ and
    $\seminorm{q}_\beta \in \cs(W_\beta)$. Using
    \eqref{eq:TwistAbelianPowers} and the Cauchy estimates from
    Lemma~\ref{lem:CauchyEstimates}, we get
    \begin{align*}
        &\Bigl(
            \seminorm{p}_{t,\seminorm{q}_\alpha}^{(1/2)}
            \tensor
            \seminorm{p}_{t,\seminorm{q}_\beta}^{(1/2)}
        \Bigr)
        \biggl(
        \frac{\hbar^n}{n!}
        r^n
        \acts
        (v \tensor w)
        \biggr) \\
        &\le
        \frac{m^n}{n!}
        \sum_{k_1,\ldots,k_n=1}^N
        \seminorm{p}_{t,\seminorm{q}_\alpha}^{(1/2)}
        (r_{k_1} \cdots r_{k_n} \acts v)
        \cdot
        \seminorm{p}_{t,\seminorm{q}_\beta}^{(1/2)}
        (s_{k_1} \cdots s_{k_n} \acts w) \\
        &\le
        (c^2 \cdot m \cdot N)^n
        \cdot
        \seminorm{p}_{2t,\seminorm{q}_\alpha}^{(1/2)}(v)
        \cdot
        \seminorm{p}_{2t,\seminorm{q}_\alpha}^{(1/2)}(w)
    \end{align*}
    for all $n \in \N_0$, $\hbar \in \Ball_m(0)^\cl$,
    $0 < 2t < \min\{t_1,t_2\}$,
    $v \in \Analytic_{1/2,t_1}(\varrho_{V,\alpha})$ and
    ${w \in \Analytic_{1/2,t_2}(\varrho_{W,\beta})}$. Now, we absorb the
    factor $T \coloneqq c^2 \cdot m \cdot N$ into $\hbar$ by
    considering the rescaled ball
    \begin{equation*}
        K
        \coloneqq
        \Ball_{m/T}(0)^\cl
        \subseteq
        \C,
    \end{equation*}
    see again Remark~\ref{rem:SmallRadius}, \ref{item:SmallRadiusCompact} and
    \eqref{eq:SmallRadiusAnalytic}. Setting $t \coloneqq 2$ and
    $C \coloneqq 1$ thus does the job.
\end{proof}

Hence, we may apply both Theorem~\ref{thm:UniversalDeformationEntire}
and Theorem~\ref{thm:UniversalDeformationAnalytic} to deform $\mu$.
\begin{remark}
    \label{remark:ParticularCaseAbelianDeformation}%
    As a particular case of such a deformation one can consider the
    symmetric algebra $\Sym(\liealg{g}^*) \cong \Pol(\liealg{g})$ of
    the dual $\liealg{g}^*$ of $\liealg{g}$, viewed as polynomials on
    $\liealg{g}$. Then $\liealg{g}$ acts by derivations on
    $\Sym(\liealg{g}^*)$ via constant vector fields on the vector
    space $\liealg{g}$.  This particular scenario was discussed in
    detail in \cite{waldmann:2014a}, including also
    infinite-dimensional abelian Lie algebras $\liealg{g}$.
\end{remark}

%% file: TeX/ExamplesAxPlusB.tex

Let $\liealg{s}$ be the two-dimensional Lie algebra with generators
$H$ and $E$ such that $[H,E] = E$. We choose a norm on $\liealg{s}$
such that $\norm{H} = 1 = \norm{E}$. By
\cite[Theorem~2.20]{giaquinto.zhang:1998a}, the formal series
\begin{equation}
    F_\hbar
    \coloneqq
    \sum_{n=0}^{\infty}
    \frac{\hbar^n}{n!}
    \cdot
    F_n
\end{equation}
with
\begin{equation}
    \label{eq:TwoDimensionalTwist}
    F_n
    \coloneqq
    \sum_{k=0}^{n}
    (-1)^k
    \binom{n}{k}
    \cdot
    (E^{n-k}
    H_{k \uparrow})
    \tensor
    (E^k
    H_{(n-k) \uparrow})
    \in
    \bigl(
    \Universal^n(\liealg{s})
    \tensor
    \Universal^n(\liealg{s})
    \bigr)\formal{\hbar}
\end{equation}
for all $n \in \field{N}_0$ constitutes a Drinfeld twist for
$\liealg{s}$. Here, we write
\begin{equation}
    \label{eq:Pochhammer}
    H_{k \uparrow}
    \coloneqq
    \prod_{j=0}^{k-1}
    (H + j)
    =
    \sum_{j=0}^{k}
    \StirlingOne{k}{j}
    H^j
    \in
    \Universal^k (\liealg{s})
    \qquad
    \textrm{for all }
    k \in \field{N}
\end{equation}
for the rising Pochhammer symbol evaluated at $H$ by means of the
polynomial calculus of the algebra $\Universal(\liealg{s}$) and
$\StirlingOne{k}{j}$ denotes the Stirling number of first kind. By
convention, we moreover set
$H_{0 \uparrow} \coloneqq 1 \in \C = \Universal^0
\liealg{s}$. Plugging \eqref{eq:Pochhammer} into
\eqref{eq:TwoDimensionalTwist} yields
\begin{equation}
    F_n
    =
    \sum_{k=0}^{n}
    (-1)^k
    \binom{n}{k}
    \sum_{j=0}^k
    \sum_{\ell=0}^{n-k}
    \StirlingOne{k}{j}
    \StirlingOne{n-k}{\ell}
    (E^{n-k} H^j)
    \tensor
    (E^k H^\ell).
\end{equation}

This example requires now $R = 1$, being the boundary case where we still can expect
interesting examples once we allow \emph{inductive} malleable
$\liealg{g}$-triples. 
\begin{proposition}
    Let $R \coloneqq 1$ and let
    \begin{equation}
        \mu
        \colon
        V \tensor W
        \longrightarrow
        X
    \end{equation}
    be a malleable inductive $\liealg{g}$-triple. Then the Drinfeld
    twist
    $F_\hbar = \sum_{n=0}^{\infty} \tfrac{\hbar^n}{n!} \cdot F_n$ with
    components \eqref{eq:TwoDimensionalTwist} fulfils the
    equicontinuity conditions
    \eqref{eq:FormalTwistEquicontinuityEntire} and
    \eqref{eq:FormalTwistEquicontinuityAnalytic}.
\end{proposition}
\begin{proof}
    Let $\alpha \in J_V$, $\beta \in J_W$, $r_1,r_2 > 0$,
    $\seminorm{q}_\alpha \in \cs(V_\alpha)$ and
    $\seminorm{q}_\beta \in \cs(W_\beta)$. By
    Remark~\ref{rem:SmallRadius},~\ref{item:SmallRadiusLowerBound}
    and \eqref{eq:SmallRadiusAnalytic} it
    suffices to establish \eqref{eq:FormalTwistEquicontinuityAnalytic}
    for $r > 0$ such that
\begin{equation*}
    r_0/2 \le r < r_0
    \qquad \textrm{with} \qquad
    r_0
    \coloneqq
    \frac{\min\{r_1,r_2\}}{4}.
\end{equation*}
For later use, we also set $N \coloneqq \lceil 4/r_0 \rceil$ as the
smallest integer larger than $4/r_0$. Using the Cauchy estimates from
Lemma~\ref{lem:CauchyEstimates} yields
\begin{align*}
    &\Bigl(
    \seminorm{p}_{r,\seminorm{q}_\alpha}^{(1)}
    \tensor
    \seminorm{p}_{r,\seminorm{q}_\beta}^{(1)}
    \Bigr)
    \biggl(
    \frac{\hbar^n}{n!}
    F_n
    \acts
    (v \tensor w)
    \biggr) \\
    &\le
    \frac{m^n}{n!}
    \sum_{k=0}^{n}
    \binom{n}{k}
    \sum_{j=0}^k
    \sum_{\ell=0}^{n-k}
    \StirlingOne{k}{j}
    \StirlingOne{n-k}{\ell}
    \seminorm{p}_{r,\seminorm{q}_\alpha}^{(1)}
    \bigl(
    E^{n-k} H^j \acts v
    \bigr)
    \cdot
    \seminorm{p}_{r,\seminorm{q}_\beta}^{(1)}
    \bigl(
    E^k H^\ell \acts w
    \bigr) \\
    &\le
    \frac{m^n}{n!}
    \sum_{k=0}^{n}
    \binom{n}{k}
    \sum_{j=0}^k
    \sum_{\ell=0}^{n-k}
    \StirlingOne{k}{j}
    \StirlingOne{n-k}{\ell}
    r^{-(n-k+j+k+\ell)}
    \cdot
    \seminorm{p}_{2r,\seminorm{q}_\alpha}^{(1)}(v)
    \cdot
    \seminorm{p}_{2r,\seminorm{q}_\beta}^{(1)}(w) \\
    &=
    \frac{m^n}{n! \cdot r^n}
    \cdot
    \seminorm{p}_{2r,\seminorm{q}_\alpha}^{(1)}(v)
    \cdot
    \seminorm{p}_{2r,\seminorm{q}_\beta}^{(1)}(w)
    \cdot
    \sum_{k=0}^{n}
    \binom{n}{k}
    \biggl(
    \sum_{j=0}^k
    \StirlingOne{k}{j}
    r^{-j}
    \biggr)
    \cdot
    \biggl(
    \sum_{\ell=0}^{n-k}
    \StirlingOne{n-k}{\ell}
    r^{-\ell}
    \biggr) \\
    &=
    \frac{m^n}{n! \cdot r^n}
    \cdot
    \seminorm{p}_{2r,\seminorm{q}_\alpha}^{(1)}(v)
    \cdot
    \seminorm{p}_{2r,\seminorm{q}_\beta}^{(1)}(w)
    \cdot
    \sum_{k=0}^{n}
    \binom{n}{k}
    \cdot
    \biggl(
    \frac{1}{r}
    \biggr)_{k \uparrow}
    \cdot
    \biggl(
    \frac{1}{r}
    \biggr)_{(n-k) \uparrow} \\
    &=
    \frac{m^n}{n! \cdot r^n}
    \cdot
    \seminorm{p}_{2r,\seminorm{q}_\alpha}^{(1)}(v)
    \cdot
    \seminorm{p}_{2r,\seminorm{q}_\beta}^{(1)}(w)
    \cdot
    \biggl(
    \frac{2}{r}
    \biggr)_{n \uparrow} \\
    &\le
    \frac{(2m)^n}{r_0^n}
    \cdot
    \seminorm{p}_{2r,\seminorm{q}_\alpha}^{(1)}(v)
    \cdot
    \seminorm{p}_{2r,\seminorm{q}_\beta}^{(1)}(w)
    \cdot
    \frac{1}{n!}
    \cdot
    \biggl(
    \frac{4}{r_0}
    \biggr)_{n \uparrow} \\
    &\le
    \frac{(2m)^n}{r_0^n}
    \cdot
    \seminorm{p}_{2r,\seminorm{q}_\alpha}^{(1)}(v)
    \cdot
    \seminorm{p}_{2r,\seminorm{q}_\beta}^{(1)}(w)
    \cdot
    \frac{(N+n-1)!}{(N-1)! \cdot n!} \\
    &\le
    2^{N-1}
    \cdot
    \frac{(4 m)^n}{r_0^n}
    \cdot
    \seminorm{p}_{2r,\seminorm{q}_\alpha}^{(1)}(v)
    \cdot
    \seminorm{p}_{2r,\seminorm{q}_\beta}^{(1)}(w)
\end{align*}
for all $n \in \N_0$, $r_0/2 < r < r_0$, $\hbar \in \Ball_m(0)^\cl$,
$v \in \Analytic_{1,r_1}(\varrho_V)$ and
$w \in \Analytic_{1,r_2}(\varrho_W)$. Hence, setting
$C \coloneqq 2^{N-1}$ and $T \coloneqq 4m/r_0$ completes the proof.
\end{proof}

Consequently, we may apply Theorem~\ref{thm:UniversalDeformationEntire} and
Theorem~\ref{thm:UniversalDeformationAnalytic} also in this situation. A~concrete
representation of $\liealg{s}$ may be defined by linearly extending
\begin{equation}
    \label{eq:Ax+BRep}
    \varrho(H)
    \coloneqq
    -
    z
    \frac{\D}{\D z}
    \qquad \textrm{and} \qquad
    \varrho(E)
    \coloneqq
    \frac{\D}{\D z},
\end{equation}
where both operators are acting continuously on the space of entire functions
$\Holomorphic(\C)$ in one variable. 
\begin{proposition}
    \label{prop:Ax+BConcreteRepresentation}%
    The representation $\varrho \colon \liealg{s} \longrightarrow
    \Linear(\Holomorphic(\C))$ from \eqref{eq:Ax+BRep} fulfils
    \begin{equation}
        \label{eq:Ax+BConcrete}
        \Entire_0(\varrho)
        =
        \Holomorphic(\C)
        \qquad \textrm{and} \qquad
        \Analytic_1(\varrho)
        =
        \Holomorphic_1(\C)
    \end{equation}
    as locally convex spaces.
\end{proposition}

 Here, $\Holomorphic_1(\C)$ denotes the space of entire functions of order at most
 one and of minimal type. We are going to explain it in detail within
 Section~\ref{subsec:ExamplesHeisenberg}.

%% file: TeX/ExamplesHeisenberg.tex

Let $d \in \field{N}$ and write $E_{nk}$ for the unit matrix with
entries
\begin{equation}
    (E_{nk})_{\ell m}
    =
    \delta_{n,\ell} \cdot \delta_{k,m}
    \qquad \textrm{for }
    1 \le n,k,\ell,m \le d,
\end{equation}
i.e. the matrix with entries one at position
$(n,k)$ and zeros otherwise. We consider the Lie subalgebra
$\mathcal{H} \subseteq \liealg{sl}_d$ generated by
the diagonal matrices
\begin{equation}
    H_n
    \coloneqq
    E_{nn}
    -
    E_{n+1,n+1}
    \qquad
    \textrm{for }
    n=1,\ldots,d-1
\end{equation}
and $E_{1n}$ as well as $E_{n1}$ for $n = 2,\ldots,d-1$. By
\cite[Theorem~2.6]{giaquinto.zhang:1998a}, a Drinfeld twist of
$\mathcal{H}$ is given by
\begin{equation}
    F_\hbar
    \coloneqq
    \sum_{n=0}^{\infty}
    \frac{\hbar^n}{n!}
    \cdot
    F_n
\end{equation}
with
\begin{equation}
    \label{eq:SpecialLinearTwist}
    F_n
    \coloneqq
    \sum_{m=0}^{n}
    \binom{n}{m}
    X^m
    \cdot
    \bigl(
    (H+m)_{(n-m) \uparrow}
    \tensor
    E_{1d}^{n - m}
    \bigr)
    \qquad
    \textrm{for all }
    n \in \field{N}_0.
\end{equation}
We are also going to need the explicit expressions
\begin{equation}
    X
    =
    \sum_{s=2}^{d-1}
    E_{1s}
    \tensor
    E_{sd}
    \qquad \textrm{and} \qquad
    H
    =
    \frac{1}{2}
    \sum_{s=1}^{d-1}
    c_s \cdot H_s,
\end{equation}
where the precise form the coefficients $c_1,\ldots,c_{d-1} \in \C$
does not matter for our purposes. We endow the space of all square
matrices $\Mat_{d}(\C)$ of size $d \in \N$ and thus also $\mathcal{H}$
with the Frobenius norm. Then $\norm{E_{nm}} = 1$ for all
$n,m = 1,\ldots,d$. Expanding the Pochhammer symbol as before in
\eqref{eq:Pochhammer} yields
\begin{equation}
    (H+m)_{(n-m) \uparrow}
    =
    \sum_{k=0}^{n-m}
    \StirlingOne{n-m}{k}
    (H+m)^k
    =
    \sum_{k=0}^{n-m}
    \sum_{\ell=0}^{k}
    \StirlingOne{n-m}{k}
    \binom{k}{\ell}
    m^{k-\ell}
    H^\ell
\end{equation}
for all $n \ge m \in \field{N}_0$. We require $R = 1$ also in this situation:
\begin{proposition}
    Let $R \coloneqq 1$ and let
    \begin{equation}
        \mu
        \colon
        V \tensor W
        \longrightarrow
        X
    \end{equation}
    be a malleable inductive $\liealg{g}$-triple. Then the Drinfeld
    twist
    $F_\hbar = \sum_{n=0}^{\infty} \tfrac{\hbar^n}{n!} \cdot F_n$ with
    components \eqref{eq:SpecialLinearTwist} fulfils the
    equicontinuity conditions
    \eqref{eq:FormalTwistEquicontinuityEntire} and
    \eqref{eq:FormalTwistEquicontinuityAnalytic}.
\end{proposition}
\begin{proof}
    Let $\alpha \in J_V$, $\beta \in J_W$, $r_1,r_2 > 0$,
    $\seminorm{q}_\alpha \in \cs(V_\alpha)$ and
    $\seminorm{q}_\beta \in \cs(W_\beta)$. Invoking
    Remark~\ref{rem:SmallRadius}, \ref{item:SmallRadiusLowerBound}
    and \eqref{eq:SmallRadiusAnalytic}, it suffices to establish
    \eqref{eq:FormalTwistEquicontinuityAnalytic}
    for $r > 0$ such that
    \begin{equation*}
        r_0/2 \le r < r_0
        \qquad \textrm{with} \qquad
        r_0
        \coloneqq
        \frac{\min\{r_1,r_2\}}{4}.
    \end{equation*}
    Moreover, we set $N \coloneqq \lceil \norm{2 H}/r_0 \rceil$. Using
    the Cauchy estimates from Lemma~\ref{lem:CauchyEstimates} yields
    \begin{align*}
        &\Bigl(
        \seminorm{p}_{r,\seminorm{q}_\alpha}^{(1)}
        \tensor
        \seminorm{p}_{r,\seminorm{q}_\beta}^{(1)}
        \Bigr)
        \biggl(
        \frac{\hbar^n}{n!}
        F_n
        \acts
        (v \tensor w)
        \biggr) \\
        &\le
        \frac{M^n}{n!}
        \sum_{m=0}^n
        \sum_{k=0}^{n-m}
        \sum_{\ell=0}^{k}
        \StirlingOne{n-m}{k}
        \binom{k}{\ell}
        m^{k-\ell}
        \sum_{s_1,\ldots,s_m=2}^{d-1}
        \seminorm{p}_{r,\seminorm{q}_\alpha}^{(1)}
        \bigl(
        E_{1s_1}
        \cdots
        E_{1s_m}
        H^\ell
        \acts
        v
        \bigr) \\
        &\hspace{8cm}
        \cdot
        \seminorm{p}_{r,\seminorm{q}_\beta}^{(1)}
        \bigl(
        E_{s_1 d}
        \cdots
        E_{s_m d}
        E_{1d}^{n-m}
        \acts
        w
        \bigr) \\
        &\le
        \sum_{m=0}^n
        \frac{M^n \cdot d^m}{n! \cdot r^{n+m}}
        \cdot
        \seminorm{p}_{2r,\seminorm{q}_\alpha}^{(1)}(v)
        \cdot
        \seminorm{p}_{2r,\seminorm{q}_\beta}^{(1)}(w)
        \cdot
        \sum_{k=0}^{n-m}
        \sum_{\ell=0}^{k}
        \StirlingOne{n-m}{k}
        \binom{k}{\ell}
        m^{k-\ell}
        \cdot
        \frac{\norm{H}^\ell}{r^\ell} \\
        &=
        \sum_{m=0}^n
        \frac{M^n}{n! \cdot r^{n+m}}
        \cdot
        \seminorm{p}_{2r,\seminorm{q}_\alpha}^{(1)}(v)
        \cdot
        \seminorm{p}_{2r,\seminorm{q}_\beta}^{(1)}(w)
        \cdot
        \biggl(
        \frac{\norm{H}}{r}
        +
        m
        \biggr)_{(n-m)\uparrow} \\
        &\le
        \sum_{m=0}^n
        \frac{M^n \cdot 2^{n+m}}{n! \cdot\min\{1,r_0\}^{n+m}}
        \cdot
        \seminorm{p}_{2r,\seminorm{q}_\alpha}^{(1)}(v)
        \cdot
        \seminorm{p}_{2r,\seminorm{q}_\beta}^{(1)}(w)
        \cdot
        (N+m)_{(n-m)\uparrow}  \\
        &=
        \sum_{m=0}^n
        \frac{M^n \cdot 2^{n+m}}{\min\{1,r_0\}^{n+m}}
        \cdot
        \frac{(N+n-1)!}{(N+m-1)! \cdot (n-m)!}
        \cdot
        \frac{(n-m)!}{n!}
        \cdot
        \seminorm{p}_{2r,\seminorm{q}_\alpha}^{(1)}(v)
        \cdot
        \seminorm{p}_{2r,\seminorm{q}_\beta}^{(1)}(w) \\
        &\le
        \sum_{m=0}^n
        \frac{M^n \cdot 4^{n}}{\min\{1,r_0\}^{2n}}
        \cdot
        2^{N+n}
        \cdot
        \seminorm{p}_{2r,\seminorm{q}_\alpha}^{(1)}(v)
        \cdot
        \seminorm{p}_{2r,\seminorm{q}_\beta}^{(1)}(w) \\
        &\le
        2^N
        \cdot
        \frac{(16M)^n}{\min\{1,r_0\}^{2n}}
        \cdot
        \seminorm{p}_{2r,\seminorm{q}_\alpha}^{(1)}(v)
        \cdot
        \seminorm{p}_{2r,\seminorm{q}_\beta}^{(1)}(w)
    \end{align*}
    for all $n \in \N_0$, $r_0/2 < r < r_0$,
    $\hbar \in \Ball_M(0)^\cl$, $v \in \Analytic_{1,r_1}(\varrho_V)$
    and $w \in \Analytic_{1,r_2}(\varrho_W)$.
\end{proof}

Motivated by \cite[Ex.~2.19~(2)]{giaquinto.zhang:1998a}, we consider
the natural Lie algebra representation of $\liealg{sl}_d$ on the space
of entire holomorphic functions~$\Holomorphic(\C^d)$ on $\C^d$. On
generators, it is given by
\begin{equation}
    \label{eq:SpecialLinearRep}
    E_{nm}
    \acts
    f
    \at[\Big]{z}
    \coloneqq
    z^n
    \cdot
    \bigl(
        \partial_m f
    \bigr)
    (z)
    \qquad
    \textrm{for }
    1 \le n,m \le d,
\end{equation}
where $f \in \Holomorphic(\C^d)$, $z = (z^1,\ldots,z^d)$ are the
standard coordinates of $\C^d$ and $\partial_m$ is the holomorphic
partial derivative in direction of $z^m$. Consider now the defining
system of norms $\{\norm{\argument}_r\}_{r > 0}$ for
$\Holomorphic(\C^d)$ given by
\begin{equation}
    \norm{f}_r
    \coloneqq
    \max_{z \in \Delta_r}
    \abs[\big]
    {f(z)}
    \qquad
    \textrm{for all }
    f \in \Holomorphic(\C^d),
\end{equation}
where
$\Delta_r \coloneqq \{z \in \C^d \;|\; \forall_{k=1,\ldots,d} \colon
\abs{z^k} \le r\}$ denotes the closed polydisk of radius
$(r,\ldots,r)$ around the origin.
\begin{proposition}
    \label{prop:SlDEntireVectors}%
    The representation
    $\varrho \colon \liealg{sl}_d \longrightarrow
    \Linear(\Holomorphic(\C^d))$ from \eqref{eq:SpecialLinearRep}
    fulfils
    \begin{equation}
        \Entire_0(\varrho)
        =
        \Holomorphic(\C^d)
    \end{equation}
    as locally convex spaces.
\end{proposition}
\begin{proof}
    Invoking the classical Cauchy estimates yields
    \begin{equation}
        \label{eq:CauchyEstimatesClassical}
        \norm[\big]
        {
          E_{nm}
          \acts
          f
        }_{r_0}
        \le
        \frac{r_0}{r}
        \cdot
        \norm{f}_{r_0 + r}
        \qquad
        \textrm{for all }
        f \in \Holomorphic(\C^d), \,
        r_0, r > 0.
    \end{equation}
    Recall that, by Pythagoras' Theorem, the basis decomposition
    \begin{equation*}
        \xi
        =
        \sum_{n,m = 1}^{d}
        \xi^{n m}
        \cdot
        E_{nm}
    \end{equation*}
    of $\xi \in \UnitBall$ fulfils $\abs{\xi^{nm}} \le 1$ for all
    $1 \le n,m \le d$. Combining both observations with all radii
    given by~$r \coloneqq (3 \cdot d^2 \cdot (r_0+r_1) \cdot r_1)/k$ for fixed $r_0,r_1
    > 0$ within \eqref{eq:CauchyEstimatesClassical} leads to
    \begin{align*}
        \sup_{\xi_1, \ldots, \xi_k \in \UnitBall}
        \norm[\big]
        {
          \xi_1 \cdots \xi_k
          \acts
          f
        }_{r_0}
        &\le
        \sup_{\xi_1, \ldots, \xi_k \in \UnitBall}
        \abs[\big]
        {\xi_1^{n_1 m_1}}
        \cdots
        \abs[\big]
        {\xi_k^{n_k m_k}}
        \cdot
        \norm[\big]
        {
          E_{n_1 m_1}
          \cdots
          E_{n_k m_k}
          \acts
          f
        }_{r_0} \\
        &\le
        d^{2k}
        \cdot
        \frac{k^k \cdot (r_0+r_1)^k}{3^k \cdot d^{2k} \cdot (r_0+r_1)^k \cdot r_1^k}
        \cdot
        \norm{f}_{r_0(1+3d^2r_1)} \\
        &=
        \frac{k^k}{(3r_1)^k}
        \cdot
        \norm{f}_{r_0(1+3d^2r_1)}
    \end{align*}
    for all $f \in \Holomorphic(\C^d)$ and $k \in \N_0$, where we have
    employed Einstein's summation convention. Consequently,
    \begin{align*}
        \seminorm{p}_{r_1,\norm{\argument}_{r_0}}^{(0)}(f)
        &=
        \sum_{k=0}^{\infty}
        \frac{r^k_1}{k!}
        \cdot
        \sup_{\xi_1, \ldots, \xi_{k} \in \UnitBall}
        \norm[\big]
        {
          \xi_1 \cdots \xi_k
          \acts
          f
        }_{r_0} \\
        &\le
        \sum_{k=0}^{\infty}
        \frac{r^k}{k!}
        \cdot
        \frac{k^k}{(3r_1)^k}
        \cdot
        \norm{f}_{r_0 (1 + 3d^2 r_1)} \\
        &=
        C
        \cdot
        \norm{f}_{r_0 (1 + 3d^2 r_1)}
    \end{align*}
    for $r_0,r_1 > 0$ and $f \in \Holomorphic(\C^d)$. Here we use that,
    by Stirling's approximation,
    \begin{equation*}
        \lim_{k \rightarrow \infty}
        \frac{k}{\sqrt[k]{k!}}
        =
        \E
        <
        3,
    \end{equation*}
    which ensures the convergence of the series
    $C \coloneqq \sum_{k=0}^{\infty} \tfrac{k^k}{k! \cdot 3^k}$. We have thus
    established the continuous inclusion
    \begin{equation}
        \Holomorphic(\C^d)
        \subseteq
        \Entire_0(\C^d),
    \end{equation}
    as $C$ is independent of $f$. Conversely, as \eqref{eq:SpecialLinearRep} is defined
    only on $\Holomorphic(\C^d)$, the equality follows. Finally, the continuity of the
    inverse inclusion is automatic by the open mapping theorem as both spaces are
    Fréchet, see again Corollary~\ref{cor:EntireCompleteness}.
\end{proof}

Turning our attention to the more relevant case of $R = 1$, we
first note that any analytic vector of order one is, in particular,
entire of all orders $R < 1$. As a candidate, we consider the algebra
$\Holomorphic_1(\C^d)$ of entire functions of order at most one and
of minimal type. Such functions are holomorphic mappings~$f \in
\Holomorphic(\C^d)$ such that
\begin{equation}
    \label{eq:FiniteOrderSeminorms}
    \seminorm{q}_r(f)
    \coloneqq
    \sup_{z \in \C^d}
    \abs[\big]
    {f(z)}
    \cdot
    \exp(-r\norm{z})
    <
    \infty
\end{equation}
for all $r > 0$. As the mappings $\seminorm{q}_r$ constitute
seminorms, this endows $\Holomorphic_1(\C^d)$ with a natural locally
convex topology, which is Fréchet. Moreover, we have the elementary
continuity estimate
\begin{equation}
    \seminorm{q}_r(f \cdot g)
    \le
    \seminorm{q}_{r/2}(f)
    \cdot
    \seminorm{q}_{r/2}(g)
    \qquad
    \textrm{for all }
    r > 0, \,
    f,g \in \Holomorphic_1(\C^d)
\end{equation}
for the pointwise multiplication, turning $\Holomorphic_1(\C^d)$ into
a Fréchet algebra. We are also going to need the seminorms given by
\begin{equation}
    \label{eq:EntireOrderOneAlternativeSeminorms}
    \seminorm{m}_r(f)
    \coloneqq
    \sum_{j=0}^{\infty}
    r^j
    \sum_{k_1,\ldots,k_j=1}^{d}
    \abs[\bigg]
    {
        \frac{\partial^j f}{\partial z^{k_1} \cdots \partial z^{k_n}}
        (1,\ldots,1)
    }
\end{equation}
with $r > 0$, which constitute another defining system of seminorms for
$\Holomorphic_1(\C^d)$. In particular, a function~$f \in \Holomorphic(\C^d)$ is
within $\Entire_1(\C^d)$ iff $\seminorm{m}_r(f) < \infty$ for all $r > 0$, see e.g.
\cite[Prop.~4.15]{heins.roth.waldmann:2023a}.
Note that our notion of order $R$ for analytic vectors corresponding to the abelian
Lie algebra action acting by partial derivatives corresponds to the reciprocal
order~$1/R$ in the classical complex analytic language. For $R=1$, this is of course
virtually invisible. A comprehensive discussion of the classical theory of entire
functions with finite order is the textbook \cite{gruman.lelong:1986a}, and
\cite{vanEijndhoven:1987a} contains a systematic discussion of the associated locally
convex topologies.
\begin{proposition}
    \label{prop:HeisenbergAnalytic}%
    The representation
    $\varrho \colon \liealg{sl}_d \longrightarrow
    \Linear(\Holomorphic(\C^d))$ from \eqref{eq:SpecialLinearRep}
    fulfils
    \begin{equation}
        \label{eq:HeisenbergAnalyticOne}
        \Entire_1(\varrho)
        \subseteq
        \Holomorphic_1(\C^d)
        \subseteq
        \Analytic_1(\varrho).
    \end{equation}
\end{proposition}
\begin{proof}
    We begin with the second inclusion. Retracing the proof of the classical Cauchy
    estimates leads to
    \begin{align*}
        \abs[\big]
        {\partial_m f(z)}
        &=
        \abs[\bigg]
        {
          \frac{1}{2 \pi \I}
          \int \limits_{\boundary \Ball_1(z^m)}
          \frac
          {f(z+w \cdot \basis{e}_m)}
          {(z^m - w)^2}
          \D w
        }
        \le
        \frac{1}{2\pi}
        \int_0^{2\pi}
        \abs[\big]
        {
          f(z + \E^{\I t} \cdot \basis{e}_m )
        }
        \D t
    \end{align*}
    for all $m=1,\ldots,d$ and $f \in \Holomorphic(\C^d)$, where
    $\Ball_1(z^m) \subseteq \C$ denotes
    the open unit disk around~$z^m$ and $\basis{e}_m \in \C^d$ is the
    $m$-th standard unit vector. Moreover, the triangle inequality
    yields
    \begin{equation*}
        \frac
        {\exp(-r_0 \norm{z})}
        {\exp(-r_0 \norm{z + \E^{\I t} \cdot \basis{e}_m})}
        =
        \frac
        {\exp(r_0 \norm{z + \E^{\I t} \cdot \basis{e}_m})}
        {\exp(r_0 \norm{z})}
        \le
        \exp(r_0)
    \end{equation*}
    for all $r_0 > 0$ and $t \in [0,2\pi]$. Putting both estimates
    together thus leads to the Cauchy-type estimate
    \begin{equation}
        \label{eq:CauchyEstimateFiniteOrder}
        \seminorm{q}_{r_0}
        \bigl(
        \partial_m f
        \bigr)
        =
        \sup_{z \in \C^d}
        \abs[\big]
        {\partial_m f(z)}
        \cdot
        \exp(-r_0 \norm{z})
        \le
        \exp(r_0)
        \cdot
        \seminorm{q}_{r_0}(f)
    \end{equation}
    for all $m=1,\ldots,d$, $r_0 > 0$ and
    $f \in \Holomorphic_1(\C^d)$. Observing that the coordinate
    functions are within $\Holomorphic_1(\C^d)$, we moreover get
    \begin{equation*}
        \seminorm{q}_{r_0}(z^n f)
        \le
        \seminorm{q}_{(1-\lambda) r_0}(z^n)
        \cdot
        \seminorm{q}_{\lambda r_0}(f)
        =
        \frac{1}{\E \cdot (1-\lambda)r_0}
        \cdot
        \seminorm{q}_{\lambda r_0}(f)
    \end{equation*}
    for all $\lambda \in (0,1)$, $n=1,\ldots,d$, $r_0 > 0$ and
    $f \in \Holomorphic_1(\C^d)$. Hence,
    \begin{equation*}
        \seminorm{q}_{r_0}
        \bigl(
            E_{nm}
            \acts
            f
        \bigr)
        \le
        \frac{1}{\E \cdot (1-\lambda)r_0}
        \cdot
        \seminorm{q}_{\lambda r_0}
        \bigl(
        \partial^m f
        \bigr)
        \le
        \frac{\exp(\lambda \cdot r_0)}{(1-\lambda)r_0}
        \cdot
        \seminorm{q}_{\lambda r_0}(f)
    \end{equation*}
    for all $\lambda \in (0,1)$, $r_0 > 0$ and
    $f \in \Holomorphic_1(\C^d)$. Setting $\lambda \coloneqq
    1/\sqrt[k]{2}$ and estimating further yields the simpler
    expression
    \begin{equation}
        \label{eq:CauchyEstimateHomogeneity}
        \seminorm{q}_{r_0}
        \bigl(
            E_{nm}
            \acts
            f
        \bigr)
        \le
        \frac{2 \cdot \exp(r_0)}{r_0}
        \cdot
        \seminorm{q}_{r_0/\sqrt[k]{2}}(f)
    \end{equation}
    for all $r_0$ and $f \in \Holomorphic_1(\C^d)$. Using basis
    expansions as in the proof of
    Proposition~\ref{prop:SlDEntireVectors} this in turn leads to
    \begin{equation*}
        \sup_{\xi_1, \ldots, \xi_k \in \UnitBall}
        \seminorm{q}_{r_0}
        \bigl(
        \xi_1 \cdots \xi_k
        \acts
        f
        \bigr)
        \le
        \biggl(
            \frac{2 \cdot d^2 \cdot \exp(r_0/2)}{r_0}
        \biggr)^k
        \cdot
        \seminorm{q}_{r_0/2}(f)
    \end{equation*}
    for all $0 < r_0 < 1$, $k \in \N_0$ and
    $f \in \Holomorphic_1(\C^d)$. Here, we have used that the function
    \begin{equation*}
        (0,\infty) \ni x \mapsto \frac{\E^x}{x}
    \end{equation*}
    attains its global minimum at $x=1$ and falls monotonically
    on $(0,1)$. Note that it indeed suffices to consider $r_0 < 1$
    by monotonicity of the seminorms
    \eqref{eq:FiniteOrderSeminorms}. Consequently,
    $f \in \Analytic_1(\varrho)$ and the inclusion is
    continuous. Assuming $f \in \Entire_1(\varrho)$, we prove
    $\seminorm{m}_r(f) < \infty$ for all $r > 0$ with the alternative
    seminorms \eqref{eq:EntireOrderOneAlternativeSeminorms}. To this
    end, fix $r > 0$ and let $f \in \Entire_1(\varrho)$. Taking
    another look at the seminorms \eqref{eq:AnalyticSeminorms} and
    \eqref{eq:FiniteOrderSeminorms}, we find $D > 0$ such that
    \begin{equation*}
        \abs[\big]
        {
            \bigl(
                E_{n_1 m_1} \cdots E_{n_k m_k}
                \acts
                f
            \bigr)
            (1)
        }
        \le
        \frac{D}{(2 \cdot r \cdot d^2)^k}
    \end{equation*}
    for all $k \in \N_0$ and
    $1 \le n_1,\ldots,n_k,m_1,\ldots,m_k \le d$. Fix now $j \in \N_0$
    and let
    \begin{equation*}
        1 \le k_1 \le \cdots \le k_j \le d
    \end{equation*}
    with $N \in \N_0$ denoting the amount of indices equal to
    one. Note the recursion
    \begin{equation*}
        E_{11}^n
        =
        (z^1 \partial_1)^n
        =
        z^1
        \partial_1
        \bigl(
            z^1 \partial_1
        \bigr)^{n-1}
        =
        (z^1)^n
        \cdot
        \partial_1^n
        +
        (z^1)^{n-1}
        \partial_1^{n-1}
        =
        (z^1)^n
        \cdot
        \partial_1^n
        +
        E_{11}^{n-1}
    \end{equation*}
    for all $n \in \N$. Hence,
    \begin{equation*}
        E_{1k_j}
        \cdots
        E_{1k_{N+1}}
        E_{11}^N
        \acts
        f
        \at[\Big]{(1,\ldots,1)}
        =
        \frac{\partial^j f}{\partial z^{k_1} \cdots \partial z^{k_n}}
        (1,\ldots,1)
        +
        E_{1k_j}
        \cdots
        E_{1k_{N+1}}
        E_{11}^{N-1}
        \acts
        f
        \at[\Big]{(1,\ldots,1)},
    \end{equation*}
    as the remaining operators do not differentiate by $z^1$. Consequently, the
    triangle inequality yields
    \begin{equation*}
        \abs[\bigg]
        {
            \frac{\partial^j f}{\partial z^{k_1} \cdots \partial z^{k_n}}
            (1,\ldots,1)
        }
        \le
        \frac{D}{(2\cdot r \cdot d^2)^j}
        \cdot
        \bigl(
            1
            +
            2 \cdot r \cdot d^2
        \bigr)
    \end{equation*}
    for all $1 \le k_1 \le k_2 \le \cdots \le k_n \le d$ and thus
    \begin{equation*}
        \seminorm{m}_r(f)
        \le
        2 \cdot D
        \cdot
        \bigl(
            1
            +
            2 \cdot r \cdot d^2
        \bigr)
        <
        \infty.
    \end{equation*}
    Variation of $r > 0$ establishes $f \in \Holomorphic_1(\C^d)$, completing the
    proof.
\end{proof}

Returning to \eqref{eq:AnalyticSeminorms} and \eqref{eq:AnalyticOfSystem}, it is
clear that one inclusion within \eqref{eq:HeisenbergAnalyticOne} remains valid after
restricting $\varrho$ to the subalgebra $\mathcal{H}$.
\begin{corollary}
    The representation $\varrho \colon \mathcal{H} \longrightarrow
    \Linear(\Holomorphic(\C^d))$ from \eqref{eq:SpecialLinearRep}
    fulfils
    \begin{equation}
        \label{eq:HeisenbergAnalyticOneRestricted}
        \Holomorphic_1(\C^d)
        \subseteq
        \Analytic_1(\varrho)
    \end{equation}
    and the inclusion is continuous.
\end{corollary}

Hence, we may apply Theorem~\ref{thm:UniversalDeformationAnalytic} to deform
$\Holomorphic_1(\C^d)$. We compute the corresponding Poisson bracket. Recall
\begin{equation}
    \label{eq:Fone}
    F_1
    =
    H
    \tensor
    E_{1d}
    +
    X
\end{equation}
in view of \eqref{eq:SpecialLinearTwist}. Writing $\mu$ for the
pointwise multiplication, the corresponding Poisson bracket is thus
given by
\begin{align*}
    \{f,g\}
    &=
    \mu
    \bigl(
    F_1
    \acts
    (f \tensor g)
    -
    F_1
    \acts
    (g \tensor f)
    \bigr)
    \\
    &=
    \frac{z^1 \cdot \partial_d g}{2}
    \sum_{s=1}^{d-1}
    c_s
    \bigl(
    z^s \partial_s
    -
    z^{s+1} \partial_{s+1}
    \bigr)
    f
    -
    \frac{z^1 \cdot \partial_d f}{2}
    \sum_{s=1}^{d-1}
    c_s
    \bigl(
    z^s \partial_s
    -
    z^{s+1} \partial_{s+1}
    \bigr)
    g
    \\
    &+
    \sum_{s=2}^{d-1}
    \bigl(
    z^1
    \partial_s
    f
    \bigr)
    \cdot
    \bigl(
    z^s
    \partial_d
    g
    \bigr)
    -
    \sum_{s=2}^{d-1}
    \bigl(
    z^1
    \partial_s
    g
    \bigr)
    \cdot
    \bigl(
    z^s
    \partial_d
    f
    \bigr)
\end{align*}
for all $f,g \in \Holomorphic_1(\C^d)$.

Finally, we return to the $ax+b$ Lie algebra $\liealg{s}$.
\begin{proof}[Of Proposition~\ref{prop:Ax+BConcreteRepresentation}]
    The first equality within \eqref{eq:Ax+BConcrete} is just
    Proposition~\ref{prop:SlDEntireVectors} for $d=1$. Similarly, the inclusion
    $\Holomorphic_1(\varrho) \subseteq \Analytic_1(\varrho)$ is
    essentially Proposition~\ref{prop:HeisenbergAnalytic} for $d=1$. Indeed, using
    that $(H,E)$ is a basis of $\liealg{s}$ combined with the Cauchy-type estimates
    \eqref{eq:CauchyEstimateFiniteOrder} and \eqref{eq:CauchyEstimateHomogeneity}
    yields
    \begin{equation*}
        \sup_{\xi_1,\ldots,\xi_k \in \UnitBall}
        \seminorm{q}_{r_0}
        \bigl(
            \xi_1 \cdots \xi_k
            \acts
            f
        \bigr)
        \le
        \biggl(
            \frac{4 \cdot \exp(r_0)}{r_0}
        \biggr)^k
        \cdot
        \seminorm{q}_{r_0}(f)
    \end{equation*}
    for all $f \in \Holomorphic_1(\C)$, $0 < r_0 \le 1$ and $k \in \N$. Here, we have
    once again used the monotonicity of the seminorms. Hence, $\Holomorphic_1(\C)
    \subseteq \Analytic_1(\varrho)$ in a continuous fashion. Conversely, we have
    \begin{equation*}
        \seminorm{q}_{r_0}(f)
        =
        \sup_{z \in \C}
        \abs[\big]
        {
            f(z)
        }
        \cdot
        \exp(-r_0 \cdot \abs{z})
        \ge
        \abs[\big]
        {
            f(1)
        }
        \cdot
        \exp(-r_0).
    \end{equation*}
    and thus
    \begin{align*}
        \seminorm{p}_{r_1,\seminorm{q}_{r_0}}^{(1)}(f)
        &=
        \sum_{n=0}^\infty
        r_1^n
        \cdot
        \sup_{\xi_1, \ldots, \xi_{n} \in \UnitBall}
        \seminorm{q}_{r_0}
        \bigl(
            \xi_1 \cdots \xi_n
            \acts
            f
        \bigr) \\
        &\ge
        \sum_{n=0}^\infty
        r_1^n
        \cdot
        \seminorm{q}_{r_0}
        \bigl(
            E^n
            \acts
            f
        \bigr) \\
        &\ge
        \exp(-r_0)
        \cdot
        \sum_{n=0}^\infty
        r_1^n
        \cdot
        \abs[\big]
        {
            f^{(n)}(1)
        } \\
        &=
        \exp(-r_0)
        \cdot
        \seminorm{m}_{r_1}(f)
    \end{align*}
    for all $f \in \Holomorphic_1(\C)$ and $r_0,r_1 > 0$. This establishes the remaining
    inclusion within \eqref{eq:Ax+BConcrete} and its continuity.
\end{proof}

%% file: TeX/Outlook.tex

Let us conclude with a brief outlook on further questions and
possible developments:

The deformation functors \eqref{eq:DeformationFunctorEntire} and
\eqref{eq:DeformationFunctorInductive} do not preserve malleability in general.
Indeed, this is essentially only true in the abelian case by \cite{rieffel:1993a}, where
one can then apply the functor again, ultimately simply adding the two deformation
parameters. This matches nicely with the exponential form of the abelian
twist \eqref{eq:ExponentialOfCommutingDerivations}.

In the non-commutative case, the appropriate symmetry arises upon also deforming
the universal enveloping algebra $\Universal(\liealg{g})$ to a quantum group by
conjugating the coproduct with the Drinfeld twist. Obtaining a strict version of this
quantum group -- or its dual version
by means of a strict $2$-cocycle -- in terms of a locally compact quantum group has been 
of great interest, resulting in recent
works such as \cite{kustermans.vaes:2003a, deCommer:2018a,
bieliavsky.gayral.neshveyev.tuset:2024a}.

The affine group we studied within Section~\ref{subsec:ExamplesAxPlusB} has been
of particular interest, as it is both structurally simple enough to admit concrete
constructions \cite{stachura:2013a, bieliavsky.gayral.neshveyev.tuset:2021a}, but
nevertheless rather complicated by virtue of the absence of compactness. It would
be interesting to investigate whether these locally \emph{compact} quantum groups
are compatible with our locally \emph{convex} constructions. However, it is not
at all obvious what the appropriate transition from the C$^*$-algebraic world into
the locally convex one should be. One approach would be to construct and study
states on our deformed algebras to produce C$^*$-algebras by means of the
associated Gelfand-Naimark-Segal constructions. It is reasonable to expect that the
classical states are not directly positive also for the star product, but also require
modification by means of the twist, see e.g. \cite{waldmann:2005b} for an overview.
Pursuing this programme will be the subject of future investigations.

Moreover, and less concretely, it would be interesting to understand the
dependence of the deformation on the twist in a more conceptual way to
have also universality with respect to the twist itself. To this end,
an intrinsic topology on the set of twists would be the first step to
study continuity properties with respect to the twist and establish a
notion of morphisms between twist in such a way that the analytic
requirements \eqref{eq:FormalTwistEquicontinuityEntire} and
 \eqref{eq:FormalTwistEquicontinuityAnalytic} are preserved automatically.
 As the latter depend on the representation, the same is to be expected
 of suitable topologies. Ultimately, one would like to achieve classification
 results now in the analytic framework.

Another promising direction would be to incorporate also
infinite-dimensional Lie algebras instead of only finite-dimensional
ones. Here the Lie algebra will itself be equipped with some locally
convex topology. From our constructions it should be fairly easy to
incorporate Banach Lie algebras, see again Remark~\ref{rem:AnalyticVectors},
\ref{item:InfiniteDimensional}. However, more interesting examples
such as Lie algebras of vector fields require to go beyond Banach space techniques.